\documentclass[journal]{IEEEtran}
\usepackage{array}
\usepackage{textcomp}
\usepackage{stfloats}
\usepackage{url}
\usepackage{verbatim}
\usepackage{cite}
\usepackage{wrapfig}
\usepackage{subcaption}
\usepackage{graphicx}
\usepackage{mathptmx}      
\usepackage{booktabs} 
\usepackage{algorithm}
\usepackage{algpseudocode}
\usepackage{latexsym}
\usepackage{amsmath,amssymb} 
\usepackage{mathtools}       
\usepackage{stmaryrd}        
\usepackage{bm}              
\usepackage{bbm}              
\usepackage{amsthm}

\providecommand{\argmin}{\operatornamewithlimits{argmin}} 
\providecommand{\limsup}{\operatornamewithlimits{limsup}} 
\providecommand{\liminf}{\operatornamewithlimits{liminf}} 
\DeclareMathOperator{\Tr}{Tr}     
\DeclareMathOperator{\Var}{Var}   
\DeclareMathOperator{\Cond}{Cond} 
\DeclareMathOperator{\diag}{diag} 
\providecommand{\N}{\mathbb{N}} 
\providecommand{\R}{\mathbb{R}} 
\providecommand{\E}{\mathbb{E}} 
\providecommand{\T}{\mathrm{T}} 
\providecommand{\ind}[1]{\ensuremath{\mathbbm{1}_{\left\{#1\right\}}}} 
\renewcommand{\geq}{\geqslant} 
\renewcommand{\leq}{\leqslant} 

\DeclarePairedDelimiterX{\inner}[2]{\langle}{\rangle}{#1, #2}
\DeclarePairedDelimiter{\norm}{\lVert}{\rVert}
\DeclarePairedDelimiter{\abs}{\lvert}{\rvert}

\usepackage[colorlinks=true,linkcolor=blue,pdfborder={0 0 0}]{hyperref}
\usepackage{cleveref}

\newtheorem{theorem}{Theorem}[]
\newtheorem{proposition}[theorem]{Proposition}
\newtheorem{corollary}[theorem]{Corollary}
\newtheorem{lemma}[theorem]{Lemma}
\theoremstyle{definition}
\newtheorem{definition}[]{Definition}

\newtheorem{assumption}[]{Assumption}

\usepackage{ifthen}

\newcommand{\markupdraft}[2]{
\ifthenelse{\equal{#1}{display}}{#2}{}
\ifthenelse{\equal{#1}{color}}{\color{#2}}{}
}

\newcommand{\newcolored}[3][]{{\markupdraft{color}{#2}#3}
\ifthenelse{\equal{#1}{}}{}{\markupdraft{display}{{\color{yellow!70!black}[#1]}}}}
\newcommand{\del}[2][]{{\markupdraft{display}{{\color{orange}[removed: ``#2''[#1]]}}}} 

\renewcommand{\del}[2]{}  

\newcommand{\indup}{\mathbb{I}_{\uparrow}}
\newcommand{\inddown}{\mathbb{I}_{\downarrow}}
\newcommand{\inds}{\mathbb{I}_{s}}
\newcommand{\indl}{\mathbb{I}_{\ell}}

\newcommand{\aup}{\alpha_{\uparrow}}
\newcommand{\adown}{\alpha_{\downarrow}}

\begin{document}

\title{
Convergence Rate of the (1+1)-ES on Locally Strongly Convex and Lipschitz Smooth Functions
}%

\author{
Daiki Morinaga,
\thanks{D. Morinaga is with the Department of Computer Science, University of Tsukuba, Tsukuba, Japan; RIKEN AIP, Japan (e-mail: morinaga@bbo.cs.tsukuba.ac.jp).}
\and Kazuto Fukuchi,
\and Jun Sakuma,
\and Youhei Akimoto\thanks{
K. Fukuchi, J. Sakuma, and Y. Akimoto are with the Institute of Systems and Information Engineering, University of Tsukuba, Tsukuba, Japan; RIKEN AIP, Japan
(e-mail: fukuchi@cs.tsukuba.ac.jp; jun@cs.tsukuba.ac.jp; akimoto@cs.tsukuba.ac.jp).}
}



\thispagestyle{empty}
\noindent\textbf{Bibliographic information:}
{\small
\begin{verbatim}
@ARTICLE{morinaga2024ieeetevc,
  author={Morinaga, Daiki and Fukuchi, Kazuto and Sakuma, Jun and Akimoto, Youhei},
  journal={IEEE Transactions on Evolutionary Computation}, 
  title={Convergence Rate of the (1+1)-ES on Locally Strongly Convex and Lipschitz Smooth Functions}, 
  year={2024},
  volume={28},
  number={2},
  pages={501-515},
  doi={10.1109/TEVC.2023.3266955}}
\end{verbatim}
}

\noindent\textbf{Erratum:}
The version of this paper published at IEEE TEVC (2024) contains some technical errors. The errors do not affect the correctness
of the theorems. They are corrected in this version. For clarity, the changes are marked in {\color{blue}
blue}.

\maketitle
\setcounter{page}{1}
\begin{abstract}
Evolution strategy (ES) is one of the promising classes of algorithms for black-box continuous optimization.
Despite its broad successes in applications, 
theoretical analysis on the speed of its convergence is limited on convex quadratic functions and their monotonic transformation.
In this study, an upper bound and a lower bound of the rate of linear convergence of the (1+1)-ES on locally $L$-strongly convex functions with $U$-Lipschitz continuous gradient are derived as $\exp\left(-\Omega_{d\to\infty}\left(\frac{L}{d\cdot U}\right)\right)$ and $\exp\left(-\frac1d\right)$, respectively.
Notably, any prior knowledge on the mathematical properties of the objective function, such as Lipschitz constant, is not given to the algorithm, whereas the existing analyses of derivative-free optimization algorithms require it.
\end{abstract}


\sloppy

\section{Introduction}
\label{intro}
\subsection{Background}

Black-box optimization (BBO) is one of the classes of optimization problem settings in which only the query of the objective function value $x \mapsto f(x)$ is available.
As the computation performance is rapidly developing, practical applications of BBO problems, such as simulation-based optimization, emerge in extensive fields, that is, geoscience, biology, topology optimization, and machine learning \cite{dong2019efficient, fujii2018cma, kriest2017calibrating, uhlendorf2012long}.

Evolution strategy (ES) is regarded as a promising class of continuous BBO algorithms,
and particularly, variants of the covariance matrix adaptation evolution strategy (CMA-ES)\cite{hansen2014principled,hansen2003reducing,hansen2001completely} are in particular demonstrated to achieve prominent performance in benchmark problems and applications \cite{dong2019efficient, fujii2018cma, kriest2017calibrating, uhlendorf2012long, varelas2018comparative}.
While the practical success of ES has been widely recognized, its theoretical foundation is not significantly robust.

The oldest variant of ES is the (1+1)-ES, which was first proposed in \cite{rechenberg1973evolution} and sophisticated in \cite{kern2004learning}.
Although the mechanism of the (1+1)-ES is simpler than that of the state-of-the-art CMA-ES, both share the essence of algorithms (e.g., the randomness of the sampling of the candidate solution and the adaptation of internal parameters), which makes mathematical analysis considerably complicated.
Therefore, one of the goals of this study is to develop a mathematical technique to analyze the class of algorithms with those properties through the analysis of the (1+1)-ES.

\emph{Linear convergence} is a class of convergence and approximately refers to the process whereby, given an objective function, the algorithm finds an $\epsilon$-neighborhood of an optimum in {$\Theta\left(\log\left(1/\epsilon\right)\right)$} number of $f$-calls.
As the previous works show that the sequences generated by the (1+1)-ES behave in a manner that can be regarded as the linear convergence \cite{akimoto2018drift,akimoto2020global,auger2013linear,jagerskupper2003analysis,jagerskupper20061+,jagerskupper2007algorithmic,morinaga2019generalized},
this study focuses on the rate of linear convergence of the (1+1)-ES.





\subsection{Related Work}


\paragraph{Convergence analysis of the (1+1)-ES}

We analyze the (1+1)-ES with success-based step-size adaptation, which slightly generalizes the 1/5-success rule proposed in \cite{kern2004learning}. The same variant is analyzed in  \cite{akimoto2018drift,auger2013linear,morinaga2019generalized,glasmachers2020global}, and a further generalized variant with a covariance matrix adaptation is analyzed in \cite{morinaga2021convergence}. 
A slightly different variant is analyzed in \cite{jagerskupper2003analysis, jagerskupper20061+, jagerskupper2007algorithmic}.

On the spherical functions $x\in\R^d\mapsto g(\norm{x}^2)$, where $g:\R\to\R$ is an arbitrary strictly increasing transformation, Akimoto et al.~\cite{akimoto2018drift} derive an expected number of function evaluations for the (1+1)-ES to find an $\epsilon$-neighborhood of the optimum, namely the \emph{first hitting time} (FHT), which is in $\Theta\left(d\cdot\log(\norm{m_0 - x^*}/\epsilon)\right)$, where $m_0\in\R^d$ is an arbitrary initial point of the (1+1)-ES and $x^*$ is the optimum. Prior to \cite{akimoto2018drift}, J\"{a}gersk\"{u}pper~\cite{jagerskupper2003analysis,jagerskupper2007algorithmic} also analyze a slightly different variant on the spherical functions.

On a subset of convex quadratic functions that can be written as
$f(x_1,\ldots,x_d) = \xi\sum_{i=1}^{d/2} x_i^2 + \sum_{j=d/2+1}^d x_j^2$, where $1/\xi\to 0$ as $d\to\infty$,
J\"{a}gersk\"{u}pper~\cite{jagerskupper20061+} show that the number of function evaluations to halve the function value is in $\Theta(d \cdot \xi)$ with an overwhelming probability.

On the general convex quadratic functions $x\mapsto g(x^\mathrm{T} H x)$ where { $H\in\R^{d\times d}$ is symmetric positive-definite and is the Hessian matrix if $g(y)=y/2$}, Morinaga et al.~\cite{morinaga2021convergence} derive an upper bound of a convergence rate of the worst case in $O\left(\exp\left(-L/\Tr(H)\right)\right)$ ($L$ is the smallest eigenvalue of $H$ and $\Tr(\cdot)$ is the trace), and also derive a lower bound of the best case in $\Omega\left(\exp\left(-1/d\right)\right)$. 
The current study is a pure expansion of \cite{morinaga2021convergence} to strongly convex and Lipschitz smooth functions, including the general convex quadratic functions.\footnote{An example of strongly convex and Lipschitz smooth functions except for a convex quadratic function is $f(x) = g(x) + h(x)$, where $g(x)$ is a convex quadratic function with a Hessian matrix whose eigenvalues are bounded in $[L+M, U-M]$ and $h(x)$ is a twice continuously differentiable function with a Hessian matrix whose eigenvalues are bounded in $[-M, M]$.}
Note that loosely speaking, if the order of the necessary number of function evaluations to converge is represented as $\Theta(\psi(N))$ for a function $\psi$ and a parameter of interest $N$, it suggests that the convergence rate is in $\Theta\left(\exp(-1/\psi(N))\right)$.

On the \emph{homogeneous} functions with all the shapes of levelsets being similar, Auger and Hansen~\cite{auger2013linear} show that the class of convergence is in \emph{linear convergence}. 


On a class of functions that includes the strongly convex and Lipschitz smooth functions and a portion of non-convex functions, Morinaga and Akimoto~\cite{morinaga2019generalized} derive a constant upper bound of the FHT and prove linear convergence.
This result is extended by Y. Akimoto et al.~\cite{akimoto2020global} to the (1+1)-ES with the adaptation of a bounded covariance matrix.
{The dependency of the FHT or the convergence rate on $d$, $L$, and $U$ is not discussed in \cite{akimoto2020global} and \cite{morinaga2019generalized}}.
    
    Golovins et al.\cite{2019gradientlesgolovins} show that on a $L$-strongly convex and $U$-Lipschitz smooth function $f$, their proposed DFO method GLD-Fast finds a solution $x$ such that $f(x) - f(x^*)\leq \epsilon$ after a number of function evaluations in $O\left(\frac{d\cdot U}{L}\cdot\log^2\left(\frac{U}{L}\right)\cdot\log\left(\frac{R}{\epsilon}\right)\right)$ with high probability.
    
    These theoretical analyses assume that the algorithm can explicitly exploit parameters of the functions, such as the Lipschitz constant $U$, while they guarantee that the class of convergence of their algorithms is in linear convergence and also offer the runtime order with dependencies of the function parameters $L$ and $U$.
    As the knowledge of such mathematical properties of the objective function cannot be expected in every case of practice when the derivative is not available,
    the current study focuses on the genuine black-box optimization algorithm, which assumes the situation that only the query of the function is available.

\subsection{Contributions}

We prove the upper and lower bound of the convergence rate of the (1+1)-ES with success-based step-size adaptation on functions that can be expressed as a composite of a strictly increasing function and a locally $L$-strongly convex function with $U$-Lipschitz continuous gradients. 
Informally, the result is stated as follows: For a sufficiently large dimension $d$, the sequence of the candidate solutions, $\{m_t\}_{t \geq 0}$, generated by the (1+1)-ES satisfies with probability one.
\begin{multline}
    \frac{1}{d} 
    \geq - \liminf_{t\to\infty} \frac{1}{t}\log\left(\frac{\norm{m_t - x^\mathrm{opt}}}{\norm{m_0 - x^\mathrm{opt}}}\right) 
    \\
    \geq - \limsup_{t\to\infty} \frac{1}{t}\log\left(\frac{\norm{m_t - x^\mathrm{opt}}}{\norm{m_0 - x^\mathrm{opt}}}\right) 
    \in \Omega\left(\frac{L}{d\cdot U}\right). 
\end{multline}
Namely, the candidate solution converges linearly (i.e., geometrically) toward $x^\mathrm{opt}$ with convergence rate (i.e., the factor of decrease of $\norm{m_t - x^\mathrm{opt}}$ in each step) no smaller than $\exp\left(-\frac{1}{d}\right)$ and no greater than $\exp\left(-\Omega\left(\frac{L}{d\cdot U}\right)\right)$. The upper bound is refined to $\exp\left(-\Omega\left(\frac{L}{\Tr(H)}\right)\right)$ if the objective function is a composite of a strictly increasing function and convex quadratic function $\frac{1}{2} x^\mathrm{T} H x$ with Hessian matrix $H$ whose eigenvalues are bounded in $[L, U]$. To the best of our knowledge, this is the first analysis to explicitly derive the convergence rate of the BBO algorithm on $L$-strongly convex function with $U$-Lipschitz continuous gradients. These results are formally stated in \Cref{thm:lowerconvergencerate,theorem:lower,cor:upper}. We evaluate the tightness of the derived bounds in numerical experiments on convex quadratic problems in \Cref{sec:exp} and we empirically observe for a large $T$
\begin{multline}
    \min\left\{ 0.1 \cdot \frac{1}{d} , 10 \cdot \frac{L}{\Tr(H)} \right\}
    \\
    \gtrapprox - \frac{1}{0.1 \cdot T}\log\left(\frac{\norm{m_{T} - x^\mathrm{opt}}}{\norm{m_{0.9 \cdot T} - x^\mathrm{opt}}}\right) 
    \gtrapprox 0.1 \cdot \frac{L}{\Tr(H)} 
\end{multline}
on three types of $H$ with varying $d \in [1, 10000]$ and varying condition number $\frac{U}{L} \in [1, 10^6]$. 

\section{Formulation}

In this section, we introduce the algorithm to be analyzed and the class of objective functions to be analyzed.
The definitions of the linear convergence and the convergence rate are formally stated.


\begin{algorithm}[H]
\caption{(1+1)-ES with success-based $\sigma$-adaptation}
\label{algo}
\begin{algorithmic}[1]
\State \textbf{input} $m_0 \in \mathbb{R}^d$, $\sigma_0 > 0$, $f: \R^d \to \R$
\State \textbf{parameter} $\aup > 1$, $\adown < 1$
\For {$t = 0,1,\dots$}
\State $x_t = m_t + \sigma_t \cdot z_t$ \text{where} $z_t \sim \mathcal{N}(0,  I)$
\If {$f\big(x_t\big) \leq f\big(m_t\big)$}
\State $m_{t+1} \leftarrow x_t$
\State $\sigma_{t+1} \leftarrow \sigma_t \cdot \aup$
\Else
\State $m_{t+1} \leftarrow m_t$
\State $\sigma_{t+1} \leftarrow \sigma_t \cdot \adown$
\EndIf
\EndFor
\end{algorithmic}
\end{algorithm}

\subsection{Algorithm}

We analyze the (1+1)-ES with success-based step-size adaptation \cite{kern2004learning} presented in \Cref{algo}. 
Given the initial search point, $m_0 \in \R^d$, and the initial step-size, $\sigma_0 > 0$, it repeats the following steps until a stopping criterion is satisfied. At each iteration $t \geq 0$, it samples a candidate solution, $x_t \in \R^d$, by adding to the current search point, $m_t$, a Gaussian noise $\sigma_t \cdot z_t$ with standard deviation $\sigma_t$. Here, $z_t \sim \mathcal{N}(0, I)$ indicates that $z$ is a $d$-dimensional random vector drawn from the standard normal distribution. If the candidate solution has a better or equally well objective function value than the current search point (i.e., $f(x_t) \leq f(m_t)$) the search point is updated by the candidate solution, and the step-size is increased by multiplying with $\aup > 1$. Otherwise, the current search point is taken over, and the step-size is decreased by multiplying with $\adown < 1$.

Adaptation of $\sigma$ is designed to maintain the probability of sampling a better or equally good candidate solution than the current search point, $\Pr_{z_t}[f(x_t) \leq f(m_t)]$, to the target probability $p_\mathrm{target}$ defined as
\begin{equation}
p_\mathrm{target} := \frac{\log(1/\adown)}{\log(\aup / \adown)} 
\label{eq:p_target-def}
.
\end{equation}
The factors of the step-size change, $\aup$ and $\adown$, are the hyper-parameters of this algorithm but not intended to be tuned for each problem. A typical choice is $\adown = \aup^{-1/4}$, so that $p_\mathrm{target} = 1/5$; hence, this approach is originally proposed as \emph{1/5-success rule} \cite{rechenberg1973evolution,kern2004learning}. The increase factor, $\aup$, is typically set to either $\exp(c)$, $\exp(c/\sqrt{d})$, or $\exp(c/d)$, where $c > 0$ is a constant.

Optimization with \Cref{algo} is formulated mathematically as the sequence of the parameters, $\{m_t\}_{t\geq 0}$ and $\{\sigma_t\}_{t\geq 0}$, and the sequence of the random vectors, $\{z_t\}_{t\geq 0}$.
Subsequently, $\ind{A}$ is the indicator function of an event $A$ (i.e., $\ind{A} = 1$ if $A$ occurs) and otherwise $\ind{A} = 0$.
\begin{definition}[Markov chain of the (1+1)-ES]\label{def:algorithm}
Let $\Theta = \R^{d+1}$ be the state space, and a state of \Cref{algo} at iteration $t$ is defined as $\theta_t = (m_t, \log(\sigma_t))$.
Let $\{z_t\}_{t \geq 0}$ be the sequence of independent and $\mathcal{N}(0, I)$-distributed random vectors.
For a measurable function $f : \R^d \to \R$, let $\theta_0 = (m_0, \log(\sigma_0))$ and $\theta_{t+1} = \theta_{t} + \mathcal{G}(\theta_t, z_t; f)$, where
\begin{equation}
\begin{split}
\mathcal{G}(\theta, z; f) := (\sigma \cdot z, \log(\aup)) \cdot &\ind{f(m + \sigma \cdot z) \leq f(m)}\\
+ (0, \log(\adown)) \cdot &\ind{f(m + \sigma \cdot z) > f(m)} .
\end{split}
\end{equation}
Let $\{\mathcal{F}_t\}_{t \geq 0}$ be the natural filtration of $\{\theta_t\}_{t \geq 0}$.
We write
\begin{equation}
\{(\theta_t, \mathcal{F}_{t})\}_{t\geq 0} = \texttt{ES}(f, \theta_0, \{z_t\}_{t \geq 0}) .
\end{equation}
\end{definition}

Because of the properties of the (1+1)-ES, we can generalize the analysis on a relatively simple convex objective function to a function class including non-convex functions.
First, the (1+1)-ES is invariant against an arbitrary strictly increasing transformation of the objective function because it is comparison-based, indicating that it requires only comparing two objective function values and the function values themselves are not used. 
Second, it is invariant against the parallel translations of the coordinate system of the search space. This allows us to assume that the optimum is located at the origin without loss of generality.
Third, because the (1+1)-ES is an elitist strategy (i.e., $f(m_{t+1}) \leq f(m_t)$ for all $t \geq 0$) the function landscape outside $S(m_0) = \{x \in \R^d : f(x) \leq f(m_0)\}$ does not matter. 
Therefore, we can assume a simple structure outside $S(m_0)$ for convenience.
These properties are formally stated in the following proposition. 
Its proof is provided in \Cref{apdx:prop:invariance}.

\begin{proposition}[Properties of the (1+1)-ES]\label{prop:invariance}
Given $\{z_t\}_{t \geq 0}$ and $\theta_0 \in \Theta$, the following hold.

(1) Let $g:\R \to \R$ be an arbitrary strictly increasing function (i.e., $g(x) < g(y) \Leftrightarrow x < y$). Subsequently, $\tilde{\theta}_t = \theta_t$ for all $t \geq 0$ for $\{(\theta_t, \mathcal{F}_{t})\}_{t\geq 0} = \texttt{ES}(f, \theta_0, \{z_t\}_{t \geq 0})$ and $\{(\tilde{\theta}_t, \tilde{\mathcal{F}}_{t})\}_{t\geq 0} = \texttt{ES}(g \circ f, \theta_0, \{z_t\}_{t \geq 0})$.

(2) Let $T: x \mapsto x - x^*$ for an arbitrary $x^* \in \R$ and define $S_T: (m, \log(\sigma)) \to (T^{-1}(m), \log(\sigma))$. Subsequently, $\tilde{\theta}_t = S_T(\theta_t)$ for all $t \geq 0$,
for $\{(\theta_t, \mathcal{F}_{t})\}_{t\geq 0} = \texttt{ES}(f, \theta_0, \{z_t\}_{t \geq 0})$ and $\{(\tilde{\theta}_t, \tilde{\mathcal{F}}_{t})\}_{t\geq 0} = \texttt{ES}(f \circ T, S_T(\theta_0), \{z_t\}_{t \geq 0})$. 

(3) Let $\tilde{f}:\R^d\to\R$ be a function such that its restriction to $S(m_0) = \{x \in \R^d : f(x) \leq f(m_0)\}$ (i.e., $\tilde{f}\rvert_{S(m_0)}$) is equivalent to $f\rvert_{S(m_0)}$ and $\tilde{f}(x) > f(m_0)$ for all $x \in \R^d\setminus S(m_0)$. Subsequently, $\tilde{\theta}_t = \theta_t$ for all $t \geq 0$ for $\{(\theta_t, \mathcal{F}_{t})\}_{t\geq 0} = \texttt{ES}(f, \theta_0, \{z_t\}_{t \geq 0})$ and $\{(\tilde{\theta}_t, \tilde{\mathcal{F}}_{t})\}_{t\geq 0} = \texttt{ES}(\tilde{f}, \theta_0, \{z_t\}_{t \geq 0})$.
\end{proposition}

\subsection{Convergence Rate}


A deterministic real-valued sequence $\{X_t\}_{t\geq 0}$ is said to convergence linearly with convergence rate $r \in (0, 1)$ if
$\lim_{t\to\infty} \abs*{X_{t+1} - X^\mathrm{opt}}/\abs*{X_t - X^\mathrm{opt}} = r$. A slightly more applicable definition is $\lim_{t\to\infty} \frac{1}{t}\log (\abs*{X_{t} - X^\mathrm{opt}} /\abs*{X_0 - X^\mathrm{opt}}) = \log(r)$.
For a stochastic real-valued sequence, various generalizations are considered. 
In this study, the definition of the convergence rate is the limit of the slope of the logarithmic distance between the current solution and the optimum averaged over time.
\Cref{def:convergence} formally defines the convergence rate of the (1+1)-ES.
In what follows, $\norm{x}$ denotes the Euclidean norm of $x \in \R^d$. 
\begin{definition}[Convergence rate]\label{def:convergence}
Let $f: \R^d \to \R$ be a measurable function with the unique global optimum at $x^\mathrm{opt}\in\R^d$, and $\{\theta_t\}_{t \geq 0}$ be the sequence of the state vectors defined in \Cref{def:algorithm} with $m_0 \neq x^\mathrm{opt}$.
If there exists a constant $\mathrm{CR}>0$ satisfying
\begin{align}
\mathrm{Pr}\left[\lim_{t\rightarrow \infty} \frac{1}{t} \log \left( \frac{\norm{ m_t - x^\mathrm{opt} }}{\norm{m_0 - x^\mathrm{opt}}} \right) = -\mathrm{CR} \right] = 1
,
\end{align}
then $\exp(-\mathrm{CR})$ is called the convergence rate of the (1+1)-ES on $f$.
The upper convergence rate $\exp(-\mathrm{CR}^\mathrm{upper})$ on $f$ and the lower convergence rate $\exp(-\mathrm{CR}^\mathrm{lower})$ on $f$ are defined as constants satisfying
\begin{align}
&\mathrm{Pr}\left[\limsup_{t \to \infty} \frac{1}{t} \log \left( \frac{\norm{ m_t - x^\mathrm{opt} }}{\norm{m_0 - x^\mathrm{opt}}} \right) = -\mathrm{CR}^\mathrm{upper} \right] = 1
,
\\
&\mathrm{Pr}\left[\liminf_{t \to \infty} \frac{1}{t} \log \left( \frac{\norm{ m_t - x^\mathrm{opt} }}{\norm{m_0 - x^\mathrm{opt}}} \right) = -\mathrm{CR}^\mathrm{lower} \right] = 1
,
\end{align}
respectively.
\end{definition}

Our main mathematical tool to derive bounds for the upper and lower convergence rate of the (1+1)-ES is the strong law of large numbers on martingales \cite{chow1967strong}.
It reduces the analysis of the limit behavior to the analysis of the expected sequence change in one step. 
The following proposition is the central tool for our analysis.
Its proof is provided in \Cref{apdx:prop:convergence_rate}.
\begin{proposition}\label{prop:convergence_rate}
Let $\{\mathcal{F}_t\}_{t \geq 0}$ be a filtration of a $\sigma$-algebra, and $\{X_t\}_{t \geq 0}$ be a Markov chain adapted to $\{\mathcal{F}_t\}_{t \geq 0}$.
Consider the following conditions:
\begin{enumerate}
\item[C1] $\exists B^\mathrm{upper} > 0$ such that $\E[X_{t+1} - X_{t} \mid \mathcal{F}_t] \leq - B^\mathrm{upper}$ for all $t \geq 0$;
\item[C2] $\exists B^\mathrm{lower} > 0$ such that $\E[X_{t+1} - X_{t} \mid \mathcal{F}_t] \geq - B^\mathrm{lower}$ for all $t \geq 0$;
\item[C3] $\sum_{t=1}^{\infty}{\color{blue}\Var[X_{t}]} / t^2 < \infty$.
\end{enumerate}
(I) If C1 and C3 hold, then we have
$\Pr\left[\limsup_{t \to \infty} \frac{1}{t} (X_t - X_0) \leq - B^\mathrm{upper} \right] = 1$.
(II) If C2 and C3 hold, then we have 
$\Pr\left[\liminf_{t \to \infty} \frac{1}{t} (X_t - X_0) \geq - B^\mathrm{lower} \right] = 1$.
\end{proposition}

\subsection{Problem}\label{subsec:problem}

We analyze the Markov chain of the (1+1)-ES defined in \Cref{def:algorithm} for the objective function falling into the following class of functions, $\mathcal{P}_{d, L, U}^{\gamma}$.
Generally, locally strongly convex and Lipschitz smooth functions around the global optimum and their strictly increasing transformations are in the focus of our analysis.

\begin{definition}[$\mathcal{P}_{d, L, U}^\gamma$]\label{def:problem}
For a measurable function $f$ with the unique global minimum at the origin, let $\mathcal{T}^{\gamma}(f)$ be the set of all functions, $\{h: \R^d \to \R\}$, that can be expressed as $h(x) = (g \circ f)(x - x^\mathrm{opt})$ for $x \in \{\tilde{x} \in \R^d : h(\tilde{x}) \leq \gamma\}$, where $g: \R\to\R$ is strictly increasing and $x^\mathrm{opt} \in \R^d$ is the unique global minimum of $h$.
Let $\mathcal{S}_{d,L,U}$ be the set of all functions satisfying \Cref{asm:L_U_2diff} and \Cref{asm:tr_cond} below. 
The functions to be analyzed in this study are $\mathcal{P}_{d, L, U}^{\gamma} = \mathcal{T}^{\gamma}(\mathcal{S}_{d, L, U}) = \{h \mid h \in \mathcal{T}^\gamma(f) \text{ for } f \in \mathcal{S}_{d, L, U}\}$.
\end{definition}

Two important properties are introduced, \emph{strong convexity} and \emph{Lipschitz smoothness}, which are often assumed to analyze the convergence rate of gradient-based approaches.
In what follows, $\inner{a}{b} = a^\T b = \sum_{i=1}^{d} a_i b_i$ denotes the inner product of $a, b \in \R^d$. 
\begin{definition}[Strong convexity]
A differentiable function $f:\R^d\to \R$ is $L$-strongly convex over $S\subseteq\R^d$ if for any $x, y \in S$,
\begin{equation}
f(x) + \langle y-x, \nabla f(x) \rangle + \frac{L}{2}\norm{x-y}^2 \leq f(y)
    .
\end{equation}
\end{definition}
\begin{definition}[Lipschitz smoothness]
A differentiable function $f:\R^d\to\R$ is $U$-Lipschitz smooth over $S\subseteq\R^d$ if for any $x, y \in S$,
\begin{equation}
f(x) + \langle y-x, \nabla f(x) \rangle + \frac{U}{2}\norm{x-y}^2 \geq f(y)
    .
\end{equation}
\end{definition}
In light of \Cref{prop:invariance}, the Markov chain of the (1+1)-ES on $h \in \mathcal{T}^{\gamma}(f)$ is identified by that on $f$.
As is formally stated in \Cref{prop:f-norm} below, the analysis of the convergence rate of the (1+1)-ES on $h$ reduces to that on $f$. 
Namely, by analyzing the (1+1)-ES on $f \in \mathcal{S}_{d,L,U}$, for which we assume a relatively simple structure, we can generalize the results to a class of objective functions, $\mathcal{P}_{d, L, U}^{\gamma}$.

The first assumption on $f$ is stated below.
\begin{assumption}\label{asm:L_U_2diff}
A function $f:\R^d \to \R$ is measurable, continuously differentiable, $L$-strongly convex, and $U$-Lipschitz smooth in $\R^d$. 
The unique global minimum of $f$ is located at the origin, $x^* := \argmin_{x\in \R^d} f(x) = 0$, and the minimum value is $f(x^*) = 0$.
\end{assumption}
Here, we remark on the assumption.
If $f$ is $L$-strongly convex over $\R^d$, then there exists a unique critical point, which is the unique global minimum. 
That is, $\norm{\nabla f(x)}=0 \Leftrightarrow x=\argmin_{y \in \R^d} f(y)$. 
The existence of the global optimum $x^*$ at the origin and $f(x^*) = 0$ simplify our analysis but do not restrict the class of functions to be analyzed because of the arbitrarity of $g$ and $x^\mathrm{opt}$ in \Cref{def:problem}.
Moreover, the $L$-strong-convexity implies the Polyak-{\L}ojasiewicz condition \cite{Karimi2016}. Together with $f(x^*) = 0$, this indicates $\sqrt{2 L f(x)} \leq \norm{\nabla f(x)}$.
If $f$ is $U$-Lipschitz smooth over $\R^d$, its gradient $\nabla f$ is $U$-Lipschitz continuous. That is, $\norm{\nabla f(x) - \nabla f(y)}\leq U\cdot \norm{x-y}$ for all $x, y\in \R^d$.

\providecommand{\Thetazero}{\Theta_{0}}
The second assumption on $f$ is a technical one, which is proved later in \Cref{cor:tr_cond} to hold for sufficiently large $d$ under \Cref{asm:L_U_2diff}.
In the following, the cumulative density function induced by the $1$-dimensional standard normal distribution, $\mathcal{N}(0, 1)$, is denoted by $\Phi$. 
{See also \Cref{subsec:First-order expansion} for further description of $\mathcal{Q}_z(\theta)$.}
\begin{assumption}\label{asm:tr_cond}
Define
\begin{equation}
\mathcal{Q}_z(\theta) = \frac{2}{\sigma^2} \left(f(m + \sigma z) - f(m) - \inner{\nabla f(m)}{ \sigma z} \right) \enspace,\label{eq:q-def}
\end{equation}
and
\begin{equation}
    \mathcal{V}_\mathrm{std} = 
    \Var[\mathcal{Q}_z(\theta)]/(\E[\mathcal{Q}_z(\theta)])^2
    \enspace.
\end{equation}
Then, for $\mathcal{N}(0, I)$-distributed random vector $z \in \R^d$, 
\begin{multline}
\underset{\theta\in\Thetazero}{\sup}
\left\{\mathcal{V}_\mathrm{std}  \right\}
\\
< \frac14 \cdot \min \left\{ \Phi\left(\frac{\kappa_{\inf}}{2}\cdot\frac{1}{\sqrt{2\pi}}\right) - \frac12 , 1 - \Phi\left( \frac{\kappa_{\inf}}{2}\cdot\frac{3}{\sqrt{2\pi}}\right)\right\} 
\label{eq:tr_cond}
\end{multline}
holds with $\kappa_{\inf} := 
\underset{\theta\in\Thetazero}{\inf}
\left\{\frac{\E[\mathcal{Q}_z(\theta)]}{\E[\mathcal{Q}_z(\theta)\cdot\ind{z_e\leq 0}]}\right\} \geq 1$, where $\Thetazero = \Theta \setminus \{\theta \mid m = 0\}$ and $z_e = \inner{z}{\nabla f(m)}/\norm{\nabla f(m)}$.
\end{assumption}

Although we actually analyze the (1+1)-ES on objective function $f$ satisfying \Cref{asm:L_U_2diff} and \Cref{asm:tr_cond} (i.e., $f \in \mathcal{S}_{d, L, U}$), we can generalize the results to that on the function in $\mathcal{P}_{d, L, U}^{\gamma}$ because of the invariance properties of the (1+1)-ES summarized in \Cref{prop:invariance}.
The following proposition states that the convergence rate of the (1+1)-ES on $h \in \mathcal{T}^{\gamma}(f)$ is derived from the convergence rate of the (1+1)-ES on the corresponding $f \in \mathcal{S}_{d,L,U}$. Moreover, the convergence rate of the (1+1)-ES is equivalent to the limit of the slope of $\log(f(m_t))$ averaged over $t$.
Its proof is provided in \Cref{apdx:prop:f-norm}.
\begin{proposition}\label{prop:f-norm}
For measurable function $f$ satisfying \Cref{asm:L_U_2diff}, let $h \in \mathcal{T}^{\gamma}(f)$ and $x^\mathrm{opt}$ be the unique global minimum of $h$.
Let
$S: (m, \log(\sigma)) \mapsto (m + x^\mathrm{opt}, \log(\sigma))$.
We define $\{(\tilde{\theta}_t, \tilde{\mathcal{F}}_{t})\}_{t\geq 0} = \texttt{ES}(h, S(\theta_0), \{z_t\}_{t \geq 0})$ and $\{(\theta_t, \mathcal{F}_{t})\}_{t\geq 0} = \texttt{ES}(f, \theta_0, \{z_t\}_{t \geq 0})$.
Subsequently,  
if $\tilde{\theta}_0 \in \Theta\setminus\{\tilde{\theta}\mid \tilde{m}= x^\mathrm{opt}\}$ satisfies $h(\tilde{m}_0) \leq \gamma$,
we obtain, with probability one (the probability is taken for $\{z_t\}_{t \geq 0}$), 
\begin{align}
\limsup_{t \to \infty} \frac{1}{t} \log \left( \frac{\norm{ \tilde{m}_t - x^\mathrm{opt} }}{\norm{\tilde{m}_0 - x^\mathrm{opt}}} \right)
&=
\limsup_{t \to \infty} \frac{1}{2t} \log \left( \frac{f(m_t)}{f(m_0)} \right)
,
\label{eq:limsup}
\\
\liminf_{t \to \infty} \frac{1}{t} \log \left( \frac{\norm{ \tilde{m}_t - x^\mathrm{opt} }}{\norm{\tilde{m}_0 - x^\mathrm{opt}}} \right)
&=
\liminf_{t \to \infty} \frac{1}{t} \log \left( \frac{\norm{m_t} }{\norm{m_0}} \right)
\label{eq:liminf}
.
\end{align}
\end{proposition}

%

\section{Lemmas}

In this section, several mathematical tools to analyze a one-step behavior of the (1+1)-ES on objective function $f$ satisfying \Cref{asm:L_U_2diff} and \Cref{asm:tr_cond} are presented. 
In the following, we let $\Thetazero = \Theta \setminus \{\theta \mid m = 0\}$. 

\subsection{First-order expansion}\label{subsec:First-order expansion}
Consider $\mathcal{Q}_z(\theta)$ defined in \eqref{eq:q-def}. 
Because $f$ is strongly convex under \Cref{asm:L_U_2diff}, we have $\mathcal{Q}_z(\theta) > 0$. We can consider that $\frac{\sigma^2}{2} \mathcal{Q}_z(\theta)$ is the remainder term of the first-order Taylor expansion of $f$ around $m$. 
If we let $m = m_t$, $\sigma = \sigma_t$ and $z = z_t$, we have $f(x_t) = f(m_t) + \inner{\nabla f(m_t)}{ \sigma_t z_t} + \frac{\sigma_t^2}{2} \mathcal{Q}_{z_t}(\theta_t)$.
This quantity is essential for analyzing the one-step behavior of the algorithm.
The following lemma shows some properties of $\mathcal{Q}_z(\theta)$ under \Cref{asm:L_U_2diff}.
Its proof is provided in \Cref{apdx:lemma:variance}.
\begin{lemma}\label{lemma:variance}
Suppose that $f$ satisfies \Cref{asm:L_U_2diff}. 
Let $\mathcal{Q}_z(\theta)$ be defined in \eqref{eq:q-def}. 
Let $z \in \R^d$ be $\mathcal{N}(0, I)$-distributed and $z_e = \inner{z}{\nabla f(m)}/\norm{\nabla f(m)}$. 
Then, for any $\theta\in\Thetazero$,
(1) $d L \leq \E[\mathcal{Q}_z(\theta)] \leq d U$;
(2) $\Var[\mathcal{Q}_z(\theta)] \leq 4 d U^2$;
(3) $\abs{ \E[\mathcal{Q}_z(\theta) \cdot \ind{z_e \leq 0}] - \E[\mathcal{Q}_z(\theta)] / 2 } \leq (2 / d)^{1/2} (U/L) \cdot \E[\mathcal{Q}_z(\theta)]$.
\end{lemma}

In light of \Cref{lemma:variance}, we can show that \Cref{asm:tr_cond} is satisfied for a sufficiently large $d$ under \Cref{asm:L_U_2diff}. The proof of the following result is provided in \Cref{apdx:cor:tr_cond}.
\begin{corollary}\label{cor:tr_cond}
Given $U \geq L > 0$, there exists an integer $D$ such that \Cref{asm:tr_cond} is satisfied for all $f$ satisfying \Cref{asm:L_U_2diff} and for all $d \geq D$. 
\end{corollary}

\subsection{Log-progress}

The log-progress of the (1+1)-ES is defined as the single step decrease of the logarithm of the objective function value, that is,
\begin{equation}
\log\left(\frac{f(m_{t+1})}{f(m_t)}\right) = \log\left( \frac{ f(m_t + \sigma_t z_t)}{ f(m_t) } \right) \ind{f(m_t + \sigma_t z_t) \leq f(m_t)}.
\end{equation}
Along with inequality $\log(x) \leq x - 1$ for $x > 0$, the next lemma offers an upper bound of the expected log-progress.
The proof is provided in \Cref{apdx:lemma:qualitygain}
\begin{lemma}[Upper bound of expected log-progress]\label{lemma:qualitygain}
Let $f$ be a function which satisfies \Cref{asm:L_U_2diff}. Let $z \in \R^d$ be $\mathcal{N}(0, I)$-distributed and $z_e = \inner{z}{\nabla f(m)}/\norm{\nabla f(m)}$.
For any $\theta\in\Thetazero$,
\begin{multline}
\E\left[\left( \frac{ f(m + \sigma z) - f(m) }{ f(m) } \right) \ind{f(m + \sigma z) \leq f(m)}\right]
\leq \\
\frac{\sigma \norm{\nabla f(m)}}{f(m)}\cdot\left(
\frac{\sigma \cdot \E\left[\mathcal{Q}_z(\theta)\cdot \ind{z_e\leq 0} \right] }{2\norm{\nabla f(m)}}
- \frac{1}{\sqrt{2\pi}}\right)
\\
\cdot \Pr[f(m + \sigma z) \leq f(m) ]
.
\label{eq:qualitygain}
\end{multline}
\end{lemma}

Using inequality $x \geq 1 - \exp(\abs{x})$ for $x \leq 0$, the following lemma argues that the log-progress is finite in expectation. Moreover, with the fact that $\Var[x] \leq \E[x^2] \leq 2 \E\left[\exp(\abs{x})\right]$, it also offers an upper bound of the variance of the log-progress.
Its proof is provided in \Cref{apdx:lemma:variancebound}.
\begin{lemma}[Variance bound of log-progress]\label{lemma:variancebound}
Suppose $f:\R^d\to\R$ satisfies \Cref{asm:L_U_2diff} and $d>3$. 
Let $z \in \R^d$ be $\mathcal{N}(0, I)$-distributed.
For any $\theta\in\Thetazero$,
\begin{multline}
\E\left[\exp\left( \abs*{\log\left( \frac{ f(m + \sigma z) }{ f(m) } \right) \ind{f(m + \sigma z) \leq f(m)} }\right)\right]
\\
\leq
\frac{U}{L}\cdot\left(1+\frac{1}{d-3}\right)
.
\label{eq:lemma:variancebound}
\end{multline}
\end{lemma}

\subsection{Success probability}

The success probability is the probability of sampling a candidate solution $x_t$ with $f(x_t) \leq f(m_t)$, that is, 
\begin{equation}
\Pr[f(x_t) \leq f(m_t)] = \Pr[f(m_t + \sigma_t \cdot z_t) \leq f(m_t)].
\end{equation}
Analyzing the success probability is important to evaluate the upper bound of the expected log-progress provided in \Cref{lemma:qualitygain}.

The following lemma derives an upper bound and a lower bound of the success probability.
The proof is provided in \Cref{apdx:lemma:successprobability}.
\begin{lemma}[Bounds of success probability]\label{lemma:successprobability}
Define the normalized step-size: 
\begin{equation}
\bar\sigma = \sigma\cdot\E[\mathcal{Q}_z(\theta)]/\norm{\nabla f(m)}
\enspace.
\end{equation}
Let $f$ be a function that satisfies \Cref{asm:L_U_2diff}.
Let $z \in \R^d$ be $\mathcal{N}(0, I)$-distributed.
Then, for any $\theta\in\Thetazero$, and $\epsilon > 0$,
\begin{multline}
\Phi\left( - \frac12 \bar\sigma\cdot(1+\epsilon)\right)
  - \frac{1}{\epsilon^2} \cdot \mathcal{V}_\mathrm{std}
\leq 
\Pr\left[f(m + \sigma z) \leq f(m) \right]
\\
\leq 
\Phi\left( - \frac12 \bar\sigma\cdot(1-\epsilon)\right)
  + \frac{1}{\epsilon^2} \cdot \mathcal{V}_\mathrm{std}
.
\label{eq:lemma:successprobability}
\end{multline}
\end{lemma}

The main message of \Cref{lemma:successprobability} is that the success probability is approximated by $ \Phi\left( - \frac12 \bar\sigma\right)$ if $\mathcal{V}_\mathrm{std} \ll 1$.
Later in \eqref{eq:tr_cond_asymptotic}, $\mathcal{V}_\mathrm{std} \ll 1$ is proved for the limit $d \to \infty$ for functions of interest.
Therefore, we can bound the success probability by using the normalized step-size $\bar\sigma$. 
The following corollary provides sufficient conditions on the normalized step-size to bound the success probability from above and below.
Its proof is provided in \Cref{apdx:cor:successprobability}.
\begin{corollary}
\label{cor:successprobability}
Let $f$ be a function which satisfies \Cref{asm:L_U_2diff}.
Let $z \in \R^d$ be $\mathcal{N}(0, I)$-distributed.
Suppose that $\theta \in \Thetazero$.

Define $B_\theta^\mathrm{high} : \left(0, \frac12\right) \to \R_{> 0}$ as
\begin{equation}
B_\theta^\mathrm{high}(q) := \sup_{
{\epsilon\in\left(\sqrt{ \frac{ 2 }{1 - 2q} \cdot \mathcal{V}_\mathrm{std} }, \infty\right)}
}
\frac{ 2 \Phi^{-1} \left(1 - \left(q + \frac{1}{\epsilon^2} \cdot \mathcal{V}_\mathrm{std}\right) \right)}{ 1 + \epsilon} .
\label{eq:cor:b_high}
\end{equation}
Then, 
$B_\theta^\mathrm{high}(q)$ is positive for any $q\in(0, \frac12)$ and $\theta\in\Thetazero$, 
and $B_\theta^\mathrm{high}(q) \leq 2 \Phi^{-1}(1-q)$ for all $q \in \left(0, \frac12\right)$.
Moreover, 
\begin{equation}
\bar\sigma
\leq B_\theta^\mathrm{high}(q)
\Rightarrow
\Pr\left[f(m + \sigma z) \leq f(m) \right] > q
.
\label{eq:cor:successprobability:large}
\end{equation}

Suppose that $\mathcal{V}_\mathrm{std}<1/2$.
Define $B_\theta^\mathrm{low}: \left( \mathcal{V}_\mathrm{std} , \frac{1}{2}\right) \to \R_{> 0}$ as
\begin{equation}
B_\theta^\mathrm{low}(q) := \inf_{
{\epsilon\in\left(\sqrt{ \frac{ 1 }{q} \cdot \mathcal{V}_\mathrm{std} }, 1\right)}
}
\frac{ 2 \Phi^{-1} \left(1 - \left(q - \frac{1}{\epsilon^2} \cdot \mathcal{V}_\mathrm{std}\right) \right)}{ 1 - \epsilon} .
\label{eq:cor:b_low}
\end{equation}
Then,
$B_\theta^\mathrm{low}(q)$ is positive for any $q\in\left( \mathcal{V}_\mathrm{std} , \frac{1}{2}\right)$ and $\theta\in\Thetazero$,
and $B_\theta^\mathrm{low}(q) \geq 2 \Phi^{-1}(1-q)$ for all $q \in \left( \mathcal{V}_\mathrm{std} , \frac{1}{2}\right)$.
Moreover,
\begin{equation}
\bar\sigma
\geq B_\theta^\mathrm{low}(q)
\Rightarrow
\Pr\left[f(m + \sigma z) \leq f(m) \right] < q
.
\label{eq:cor:successprobability:small}
\end{equation}
\end{corollary}

Finally, we arrange a tool to bound the success probability by some constants. The state space $\Thetazero$ is divided into three non-exclusive subsets based on the normalized step-size $\bar\sigma$.
The proof is provided in \Cref{apdx:lemma:cases}.
\begin{lemma}[Three cases of $\sigma$]\label{lemma:cases}
Let $f$ be a function which satisfies \Cref{asm:L_U_2diff}.
Let $z \in \R^d$ be $\mathcal{N}(0, I)$-distributed.
Let $B_\theta^\mathrm{high}$ and $B_\theta^\mathrm{low}$ be defined in \Cref{cor:successprobability}.
Define $B_{\inf}^\mathrm{high}(q) = \underset{\theta\in\Thetazero}{\inf} \left\{ B_\theta^\mathrm{high}(q) \right\}$ and $B_{\sup}^\mathrm{low} = \underset{\theta\in\Thetazero}{\sup} \left\{ B_\theta^\mathrm{low}(q) \right\}$.
Then, the following statements hold.

(i) The open interval $I_q\subseteq (0, 1/2)$ defined as 
\begin{multline}
I_q := \left(
\max\left\{
\underset{\theta\in\Thetazero}{\sup}
\left\{\mathcal{V}_\mathrm{std}\right\}
,
\right.
\right.
\\
\left.
\left.
\sup\left\{
q\in\left(0, \frac12\right):
\kappa_{\inf}
\cdot
\sqrt{\frac{2}{\pi}} \leq 
B_{\sup}^\mathrm{low}(q) 
\right\}
\right\}
, \frac{1}{2} 
\right)
\label{eq:def_Iq}
\end{multline}
is nonempty.

(ii) For any $q^\mathrm{low}\in I_q$, the open interval $I_q^{\mathrm{high}}(q^\mathrm{low}) \subseteq (0, 1/2)$ defined as
\begin{multline}
I_q^{\mathrm{high}}(q^\mathrm{low})
\\
:=
\left(
\underset{\theta\in\Thetazero}{\sup}
\left\{
\sup
\left\{
q :
B_\theta^\mathrm{low}(q^\mathrm{low}) 
\leq \frac{\aup}{\adown}\cdot B_\theta^\mathrm{high}(q)
\right\}
\right\}, \frac12 \right)
\label{eq:def_qhigh}
\end{multline}
is nonempty.

(iii)
\begin{equation}
Q := 
\inf \left\{ q : 
B_{\inf}^\mathrm{high}(q)
< 
B_{\sup}^\mathrm{low}(q^\mathrm{low})
\right\}
>0
\label{eq:def_Q}
\end{equation}
exists
for any $q^\mathrm{low}\in I_q$.

Let $q^\mathrm{low} \in I_q$ and $q^\mathrm{high} \in I_q^\mathrm{high}(q^\mathrm{low})$. Then, the following statements hold.
\begin{align}
  \text{(iv) }
  &\Pr\left[f(m + \sigma z) \leq f(m) \right] \geq q^\mathrm{high} 
  \ \text{if}\ \bar\sigma < B_{\inf}^\mathrm{high}(q^\mathrm{high}). \label{eq:lemma:toosmall}  
  \\
  \text{(v) }
  &\Pr\left[f(m + \sigma z) \leq f(m) \right] \leq q^\mathrm{low} 
  \ \text{if}\  \bar\sigma > B_{\sup}^\mathrm{low}(q^\mathrm{low}). \label{eq:lemma:toolarge}
  \\
  \text{(vi) }
  &\Pr\left[f(m + \sigma z) \leq f(m) \right] \geq Q 
  \ \text{if}\ \bar\sigma \leq B_{\sup}^\mathrm{low}(q^\mathrm{low}). \label{eq:lemma:reasonable}
\end{align}

\end{lemma}



\section{Convergence Rate}

Our main results, an upper bound of the upper convergence rate and a lower bound of the lower convergence rate of the (1+1)-ES (\Cref{def:algorithm}) on objective function $h \in \mathcal{P}_{d, L, U}^{\gamma}$ (\Cref{def:problem}), are proved in this section. 
The proofs rely on the strong law of large numbers for martingale summarized in \Cref{prop:convergence_rate}.
The analysis of the convergence rate comes down to the analysis of a single-step behavior of the (1+1)-ES.
In light of \Cref{prop:f-norm}, it is sufficient to consider objective function $f$ satisfying \Cref{asm:L_U_2diff} and \Cref{asm:tr_cond} (i.e., $f \in \mathcal{S}_{d, L, U}$).

\subsection{Potential function}\label{subsec:potentialfunction}

The first step of the analysis of the upper convergence rate is to define a potential function of the state of the (1+1)-ES, $V:\Thetazero \to \R$, to measure the progress toward the optimum.
If we construct a potential function such that $\log(f(m))\leq V(\theta)$ holds for all $\theta \in \Thetazero$, then we have for any $\theta_0 \in \Thetazero$,
\begin{equation}
\underset{t\to\infty}{\limsup}\frac{1}{t}\log\left(\frac{f(m_t)}{f(m_0)}\right)
\leq
\underset{t\to\infty}{\limsup}\frac{1}{t}\log\left(\frac{V(\theta_t)}{V(\theta_0)}\right)
.
\end{equation}
Therefore, in light of \Cref{prop:f-norm}, we can obtain the upper convergence rate of the (1+1)-ES on $h \in \mathcal{T}^{\gamma}(f)$ by analyzing the convergence of $V(\theta_t)$. 
It results in analyzing the drift of $X_t = V(\theta_t)$ in light of \Cref{prop:convergence_rate}.

Before constructing a potential function, we explain the motivation of introducing a potential function instead of analyzing $\log(f(m))$ directly.
It has been known empirically and theoretically that the (1+1)-ES exhibits two different phases through the optimization on several functions \cite{akimoto2018drift,akimoto2020global,morinaga2019generalized}.
One is that the step-size $\sigma$ is well adapted and $\log(f(m))$ decreases steadily. 
The other is that the step-size $\sigma$ is on the way to a desirable range and $\log(f(m))$ has almost no progress subsequently.
In fact, when trying to evaluate the drift of $X_t = \log(f(m_t))$ using \Cref{lemma:qualitygain}, it is easy to see that the upper bound of the expected log-progress is non-negative if $\sigma / \norm{\nabla f(m)}$ is too large. 
Moreover, the expected log-progress diminishes if $\sigma / \norm{\nabla f(m)}$ approaches zero.
Hence, the drift of $X_t = \log(f(m_t))$ cannot be upper-bounded by a negative constant for all $\theta \in \Thetazero$. 
In related works \cite{akimoto2018drift,akimoto2020global,morinaga2019generalized}, this issue has been mitigated by constructing a potential function that penalizes the situation that the step-size is not well-adapted.

The potential function used in this work is defined below. The basic idea is borrowed from \cite{morinaga2019generalized,morinaga2021convergence}, but the constants are revised to obtain the desired convergence rate bound.

\begin{definition}[Potential function]\label{def:potential}
Consider the (1+1)-ES solving $f \in \mathcal{S}_{d,L,U}$ with $\aup$ and $\adown$ satisfying $p_\mathrm{target} \in I_q$, where $I_q$ is defined in \eqref{eq:def_Iq}.
Then, in light of \Cref{lemma:cases}, we can choose 
constants $q^\mathrm{low}\in I_q$ and $q^\mathrm{high}\in I_q^\mathrm{high}(q^\mathrm{low})$
so that  
\begin{align}
 q^\mathrm{low}<p_\mathrm{target}<q^\mathrm{high} .
 \label{eq:ptarget_condition}
\end{align}
For $\theta\in\Thetazero$, the potential function $V(\theta)$ for the (1+1)-ES solving $f$ is defined as
\begin{multline}
V(\theta)
= \log\left(f(m)\right) \\
+ v \cdot \log^+ \left(\frac{s\cdot\sqrt{Lf(m)}}{\sigma E_\mathcal{Q} }\right) 
+  v \cdot \log^+ \left(\frac{\sigma E_\mathcal{Q}}{\ell\cdot\sqrt{Lf(m)}}\right) ,
\label{eq:potential}
\end{multline}
where 
$\log^+(x) := \log(x) \cdot \ind{x\geq 1}$, 
$s = \sqrt{ 2 } \aup\cdot B_{\inf}^\mathrm{high}(q^\mathrm{high})$,
$\ell = \sqrt{2} \adown\cdot 
B_{\sup}^\mathrm{low}(q^\mathrm{low})$,
$E_\mathcal{Q} =
\sup_{\theta\in\Thetazero} \left\{\E[\mathcal{Q}_z(\theta)]\right\}$,
and
\begin{equation}
v = \min\left\{ \frac{w}{4 \log(\aup / \adown)} , 1 \right\}
\label{eq:v}
\end{equation}
with
\begin{equation}
w = 
\frac{L}{E_\mathcal{Q}}
\cdot
\frac{
B_{\inf}^\mathrm{high}(q^\mathrm{high})
}{\kappa_{\inf}}
\cdot
\left(
\sqrt{\frac{2}{\pi}}\cdot\kappa_{\inf}
-
B_{\sup}^\mathrm{low}(q^\mathrm{low})
\right)
\cdot Q .
\label{eq:w}
\end{equation}
\end{definition}

The difference of the potential function values in consecutive steps is bounded in several ways depending on the current state. 
The proof is provided in \Cref{apdx:lemma:potential_decrease}.
\begin{lemma}[Potential bound]\label{lemma:potential_decrease}
Consider the (1+1)-ES solving $f \in \mathcal{S}_{d,L,U}$ with $\aup$ and $\adown$ satisfying $p_\mathrm{target} \in I_q$. Let $\{\theta_t\}_{t\geq 0}$ be a sequence of the parameters of the (1+1)-ES.
Suppose $\theta_t\in\Thetazero$. 
Then,
\begin{equation}
V(\theta_{t+1}) - V(\theta_t)\leq 
\left(1-\frac{v}{2}\right)
\cdot \log\left(\frac{f(m_{t+1})}{f(m_t)}\right) + v\cdot\log\left(\frac{\aup}{\adown}\right)
,
\label{eq:general_potential_decrease}
\end{equation}
and
\begin{equation}
V(\theta_{t+1}) - V(\theta_t)
>
\left(1+v\right)\cdot\log\left(\frac{f(m_{t+1})}{f(m_t)}\right)
-2v\cdot\log\left(\frac{\aup}{\adown}\right)
.
\label{eq:constant_lower_bound_of_V}
\end{equation}
In addition, the following statements holds.

(i) If $\sigma_t < \frac{s \cdot \sqrt{ L f(m_t) }}{ \aup\cdot E_\mathcal{Q}}$, letting $\inddown := \ind{f(m_t + \sigma_t\cdot z_t) > f(m_t)}$, we have
\begin{equation}
V(\theta_{t+1}) - V(\theta_t) \leq
v \cdot\left( \log\left(\frac{\aup}{\adown}\right)\cdot\inddown - \log(\aup)\right)
.
\label{eq:small_sigma_potential_decrease}
\end{equation}

(ii) If $\sigma_t > \frac{\ell \cdot \sqrt{L f(m_t)} }{ \adown\cdot E_\mathcal{Q}}$, letting $\indup := \ind{f(m_t + \sigma_t\cdot z_t) \leq f(m_t)}$, we have
\begin{equation}
V(\theta_{t+1}) - V(\theta_t)\leq 
v \cdot\left( \log\left(\frac{\aup}{\adown}\right)\cdot\indup - \log\left(\frac{1}{\adown}\right)\right)
.
\label{eq:large_sigma_potential_decrease}
\end{equation}
\end{lemma}

Finally, we derive condition $C1$ of \Cref{prop:convergence_rate} for each of the three situations: ($i$) too small step-size situation, where the progress is made by increasing the step-size; ($ii$) too large step-size situation, where the progress is made by decreasing the step-size; and ($iii$) reasonable step-size situation, where the progress is made by moving toward the optimum.
Its proof is provided in \Cref{apdx:cor:expected_potential_decrease}.
\begin{corollary}[Expected potential decrease]\label{cor:expected_potential_decrease}
Consider the (1+1)-ES solving $f \in \mathcal{S}_{d,L,U}$ with $\aup$ and $\adown$ satisfying $p_\mathrm{target} \in I_q$. Let $\{\theta_t\}_{t\geq 0}$ be a sequence of the parameters of the (1+1)-ES.
Suppose $\theta_t \in\Thetazero$. 
Then, the following statements hold.

(i) If $\sigma_t < \frac{s \cdot \sqrt{ L f(m_t) }}{ \aup\cdot E_\mathcal{Q}}$,
\begin{multline}
\E\left[V(\theta_{t+1}) - V(\theta_t) |\mathcal{F}_t\right]
\\
\leq
\min
\left\{
    \frac{w}{4}, \log\left(\frac{\aup}{\adown}\right)
\right\}
\cdot
    (p_\mathrm{target} - q^\mathrm{high})
<0
.
\label{eq:small_sigma_expected_potential_decrease}
\end{multline}

(ii) If $\sigma_t > \frac{\ell \cdot \norm{\nabla f(m_t)} }{ \sqrt{2}\adown\cdot \E_{z_t}[\mathcal{Q}_{z_t}(\theta_t)]}$,
\begin{multline}
\E\left[V(\theta_{t+1}) - V(\theta_t) |\mathcal{F}_t\right]
\\
\leq
\min
\left\{
    \frac{w}{4}, \log\left(\frac{\aup}{\adown}\right)
\right\}
\cdot
    (q^\mathrm{low} - p_\mathrm{target})
<0
.
\label{eq:large_sigma_expected_potential_decrease}
\end{multline}

(iii) If  $ \frac{s \cdot \sqrt{ L f(m_t) }}{ \aup\cdot E_\mathcal{Q}} \leq \sigma_t \leq \frac{\ell \cdot \norm{\nabla f(m_t)} }{ \sqrt{2}\adown\cdot \E_{z_t}[\mathcal{Q}_{z_t}(\theta_t)]}$,
\begin{equation}
\E\left[V(\theta_{t+1}) - V(\theta_t) |\mathcal{F}_t\right]
\leq -\frac{w}{4} 
<0
.
\label{eq:reasonable_sigma_expected_potential_decrease}
\end{equation}
\end{corollary}

Condition $C3$ of \Cref{prop:convergence_rate} is obtained in the following lemma.
Its proof is provided in \Cref{apdx:lemma:v_variance_bound}.
\begin{lemma}[Variance bound of the potential function]\label{lemma:v_variance_bound}
Consider the (1+1)-ES solving $f \in \mathcal{S}_{d,L,U}$ with $d > 3$ and $\aup$ and $\adown$ satisfying $p_\mathrm{target} \in I_q$. Let $\{\theta_t\in\Thetazero\}_{t\geq 0}$ be a sequence of the parameters of the (1+1)-ES.
Then, $\sum_{t=1}^{\infty}{\color{blue}\Var[V(\theta_{t})]} / t^2 < \infty$.
\end{lemma}

\subsection{Upper bound of the upper convergence rate}

An upper bound of the upper convergence rate of the (1+1)-ES on objective function $h\in\mathcal{P}_{d, L, U}^\gamma$ (\Cref{def:convergence}) is derived by using \Cref{prop:convergence_rate} and \Cref{cor:expected_potential_decrease,lemma:v_variance_bound}.

\begin{theorem}[Upper convergence rate bound]\label{theorem:lower}
Let $\{(\tilde{\theta}_t, \mathcal{F}_t)\}_{t \geq 0} = \texttt{ES}(h, \tilde{\theta}_0, \{z_t\}_{t \geq 0})$ be the state sequence of the (1+1)-ES solving $h\in\mathcal{P}_{d, L, U}^\gamma$ with $d > 3$, where $\tilde{\theta}_0 \in \Theta\setminus\{\tilde{\theta}\mid \tilde{m}= x^\mathrm{opt}\}$ and $h(\tilde{m}_0) \leq \gamma$.
Suppose that $\aup$ and $\adown$ are set so that $p_\mathrm{target} \in I_q$.
%
Then, the upper convergence rate on $h$ satisfies $\exp(-\mathrm{CR}^\mathrm{upper}) \leq \exp(-B_{d, L, U}^\mathrm{upper})$, where
\begin{multline}
B_{d, L, U}^\mathrm{upper}
= \frac12\cdot\sup_{q^\mathrm{low}, q^\mathrm{high}}
\min\left\{\frac{w}{4}, \log\left(\frac{\aup}{\adown}\right)\right\}
\\
\cdot \min\{p_\mathrm{target} -  q^\mathrm{low}, q^\mathrm{high} - p_\mathrm{target} \} > 0 ,
\label{eq:B}
\end{multline}
where $\sup_{q^\mathrm{low}, q^\mathrm{high}}$ is taken over $q^\mathrm{low}\in I_q$ and $q^\mathrm{high}\in I_q^\mathrm{high}(q^\mathrm{low})$ satisfying \eqref{eq:ptarget_condition} with 
$I_q^\mathrm{high}$ defined in \eqref{eq:def_qhigh}, and $w$ is defined in \eqref{eq:w}.
%
\end{theorem}

\begin{proof}[Proof of \Cref{theorem:lower}]
Let $h \in \mathcal{T}^{\gamma}(f)$ and $f \in \mathcal{S}_{d, L, U}$ as defined in \Cref{def:problem}.
Let $S_T : (x, \log(\sigma)) \mapsto (x - x^\mathrm{opt}, \log(\sigma))$ and $\{(\theta_t, \mathcal{F}_t)\}_{t \geq 0} = \texttt{ES}(f, S_T(\tilde{\theta}_0), \{z_t\}_{t \geq 0})$ be the state sequence of the (1+1)-ES solving $f$.
Because the singleton $\{x^*\}$ is of zero Lebesuge measure, given $m_0 \neq x^*$, we have $m_t \neq x^*$ for all $t$ with probability one.
In light of \Cref{prop:f-norm}, the upper convergence rate on $h$, $-\mathrm{CR}^\mathrm{upper}$ is equal to
$\limsup_{t \to \infty} \frac{1}{2t} \log\left(\frac{f(\theta_t)}{f(\theta_0)}\right)$
with probability one.

As shown in the proof of \Cref{cor:expected_potential_decrease}, $w$ in \eqref{eq:w} is positive since $q^\mathrm{low}\in I_q$ satisfies $\sqrt{\frac{2}{\pi}}\cdot\kappa_\mathrm{inf}>B_\mathrm{sup}^\mathrm{low}(q^\mathrm{low})$ in light of \Cref{lemma:cases}.
Therefore, it is clear that $v$ defined in \eqref{eq:v} and hence, $B_{d, L, U}^\mathrm{upper}$ defined in \eqref{eq:B} are positive.

Let $X_t = V(\theta_t)$ in \Cref{prop:convergence_rate} with $V$ defined in \Cref{def:potential}.
Condition C1 --- $\E[V(\theta_{t+1}) - V(\theta_t)\mid\mathcal{F}_t]\leq -B$ for all $t\geq 0$ --- is satisfied with $B = 2\cdot B_{d, L, U}^\mathrm{upper}$ in light of \Cref{cor:expected_potential_decrease}.
Condition C3 --- $\sum_{t=1}^{\infty}\Var[V(\theta_{t}) \mid \mathcal{F}_{t-1}] / t^2 < \infty$ --- is satisfied for $d>3$ in light of \Cref{lemma:v_variance_bound}.
Therefore, with probability one, we obtain
$\limsup_{t \to \infty} \frac{1}{t} \left( V(\theta_t) - V(\theta_0) \right) \leq - 2\cdot B_{d, L, U}^\mathrm{upper}$.
Because $\log(f(m)) \leq V(\theta)$ for any $\theta \in \Thetazero$, we have
$\limsup_{t \to \infty} \frac{1}{2 t} \log\left(\frac{f(\theta_t)}{f(\theta_0)}\right) \leq - B_{d, L, U}^\mathrm{upper}$.
Hence, we obtain $\exp(-\mathrm{CR}^\text{upper}) \leq \exp(- B_{d, L, U}^\mathrm{upper}) < 1$.
\end{proof}

If $d$ is sufficiently large and the hyper-parameters $\aup$ and $\adown$ are selected so that $\Phi(- \sqrt{2/\pi}) < p_\mathrm{target} < 1/2$, the bound of the upper convergence rate is $\exp(- B_{d,L,U}^\mathrm{upper}) \in \exp\left(-\Omega\left(\min\left\{ \frac{L}{E_\mathcal{Q}},  \log\left(\frac{\aup}{\adown}\right)\right\}\right)\right)$.
\begin{corollary}[Scaling of the Upper Convergence Rate]\label{cor:upper}
Let $q^\mathrm{low} \in \left(\Phi\left(-\sqrt{\frac{2}{\pi}}\right), \frac{1}{2}\right)$ and 
$q^\mathrm{high} \in \left(
\Phi\left( \frac{\adown}{\aup} \Phi^{-1}\left( q^\mathrm{low}\right)\right)
, \frac{1}{2}\right)$, which are independent of $d$. 
Suppose that $\aup$ and $\adown$ are set so that $q^\mathrm{low} < p_\mathrm{target} < q^\mathrm{high}$ for all $d$ ($\aup$ and $\adown$ can depend on $d$). 
Then, there exists $D > 3$ such that for any $d \geq D$, (i) \Cref{asm:tr_cond} is satisfied under \Cref{asm:L_U_2diff}, (ii)
$q^\mathrm{low} \in I_q$ and $q^\mathrm{high} \in I_q^\mathrm{high}(q^\mathrm{low})$, and (iii) the following statement holds
\begin{equation}
\lim_{d \to \infty}
\left\{
\inf_{h \in \mathcal{P}_{d, L, U}^\gamma}
\left\{
\frac{ B_{d, L, U}^\mathrm{upper} }{ \min\left\{ \frac{L}{E_\mathcal{Q}},  \log\left(\frac{\aup}{\adown}\right)\right\} }
\right\}
\right\}
\in (0, \infty) ,
\label{eq:order_B}
\end{equation}    
where $B_{d, L, U}^\mathrm{upper}$ is as defined in \eqref{eq:B}.
Moreover, we have $\frac{L}{E_\mathcal{Q}} \geq \frac{L}{dU}$. If $h \in \mathcal{T}^\gamma(f)$ for $f(x) = \frac{1}{2} x^\mathrm{T} H x$, where $H$ is a symmetric matrix whose eigenvalues are bounded in $[L, U]$, we have $\frac{L}{E_\mathcal{Q}} = \frac{L}{\Tr(H)}$.
\end{corollary}
{
    We remark on the consequences.
    The hyper-parameters $\aup$ and $\adown$ are typically chosen such that $\log(\aup / \adown)\in\Omega_{d \to \infty}(1/d)$.
    In light of \Cref{cor:upper}, the upper convergence rate $\exp\left(-\mathrm{CR}^\mathrm{upper}\right)$ is in $ \exp\left(- \Omega_{d \to \infty}\left(\frac{L}{E_\mathcal{Q}} \right)\right)$ under a choice such that $\log(\aup / \adown) \in \Omega_{d \to \infty}(1/d)$.
    Otherwise (i.e. if $\log(\aup / \adown) \in o_{d \to \infty}(1/d)$), we have $\exp\left(-\mathrm{CR}^\mathrm{upper}\right) \in\exp\left(- \Omega_{d \to \infty}\left(\log\left(\frac{\aup}{\adown}\right) \right)\right) = \exp\left(-\Omega_{d \to \infty}\left(\frac{\log\left( 1/\adown\right)}{p_\mathrm{target}}\right)\right)$.
    This is rather intuitive for the following reason. $\norm{m_t - x^*}$ does not converge faster than $\sigma_t$ because $\sigma_t$ must be proportional to $\norm{m_t - x^*}$ to produce a sufficient decrease. The speed of the decrease in $\sigma_t$ is $\adown$. Therefore, the upper convergence rate should not be smaller than $\adown$.

The dependency of the convergence rate on the trace of the Hessian matrix is derived under the optimal step-size situation for a different ES variant \cite{akimoto2020tcs}. It has also been derived for a different ES variant in \cite{beyer2013dynamics} though it is not a rigorous convergence rate analysis.}
\begin{proof}
The first claim, (i), is proved in \Cref{cor:tr_cond}. In the following, we prove claims (ii) and (iii).

Continuing from the proof of \Cref{theorem:lower}, we consider $f$ corresponding to $h$. Let $\{\theta_t\}$ be defined in the proof of \Cref{theorem:lower}.

In light of \Cref{lemma:variance},
given any $0 < L \leq U$,
we have
\begin{equation}
\underset{d\to\infty}{\lim}
\left\{
\underset{f\in\mathcal{S}_{d, L, U}}{\sup}
\left\{
\mathcal{V}_\mathrm{std} 
\right\}
\right\}
\leq 
\underset{d\to\infty}{\lim}
\left\{
\frac{4U^2}{dL^2}
\right\}
= 0
.
\label{eq:tr_cond_asymptotic}
\end{equation}
Then, we have
$
\underset{d\to\infty}{\lim}
\left\{
B_{\inf}^\mathrm{high}(q)
\right\}
= 2\cdot\Phi^{-1}(1-q)
$
and
$
\underset{d\to\infty}{\lim}
\left\{
B_{\sup}^\mathrm{low}(q)
\right\}
= 2\cdot\Phi^{-1}(1-q)
$.
Moreover, as we prove in the proof of \Cref{cor:tr_cond}, we have
$\kappa_{\inf} \to 2$ as $d\to\infty$.
Therefore, the lower limit of $I_q$ defined as \eqref{eq:def_Iq} approaches $ \Phi\left(-\sqrt{\frac{2}{\pi}}\right)$ as $d\to\infty$ and the lower limit of $I_q^\mathrm{high}(q^\mathrm{low})$ approaches
$\Phi\left( \frac{\adown}{\aup} \Phi^{-1}\left( q^\mathrm{low}\right)\right)$. 
Hence, there exists an integer $D'$ such that claim (ii) holds for all $d \geq D'$.

Finally, we prove claim (iii). 
From \eqref{eq:def_Q}, it is clear that $Q \to q^\text{low}$ as $d\to\infty$.
Therefore, as $d \to \infty$, we obtain
\begin{multline}
\frac{\inf_{f \in \mathcal{S}_{d,L,U}} \{w\}}{L/E_\mathcal{Q}} 
\to 
\Phi^{-1}(1- q^\mathrm{high})
\\
\cdot\left(\sqrt{\frac{2}{\pi}} - \Phi^{-1}(1-q^\mathrm{low})\right) \cdot q^\mathrm{low}
\in (0, \infty).\qedhere
\end{multline}
\end{proof}

\subsection{Lower Convergence Rate Bound}

The lower bound of the lower convergence rate can also be derived.

\begin{theorem}[Lower Convergence Rate Bound]\label{thm:lowerconvergencerate}
Let $\{(\tilde{\theta}_t, \mathcal{F}_t)\}_{t \geq 0} = \texttt{ES}(h, \tilde{\theta}_0, \{z_t\}_{t \geq 0})$ be the state sequence of the (1+1)-ES solving $h \in \mathcal{P}_{d,L,U}^\gamma$ with $d > 3$. Suppose that $\tilde{\theta}_0 \in \Theta\setminus\{\tilde{\theta}\mid \tilde{m}= x^\mathrm{opt}\}$.
Then, the lower convergence rate on $h$ satisfies $\exp(-\mathrm{CR}^\mathrm{lower}) \geq \exp(-1/d)$.
\end{theorem}
\begin{proof}
We prove it by using \Cref{prop:f-norm,prop:convergence_rate}. 
Continuing from the proof of \Cref{theorem:lower}, we consider $f$ corresponding to $h$. Let $\{\theta_t\}$ be defined in the proof of \Cref{theorem:lower}.

We let $X_t = \log(\norm{m_t})$ in \Cref{prop:convergence_rate}. 
Lemma 4.6 of \cite{akimoto2020global} shows that
\begin{multline}
    \E[\log(\norm{m_{t+1}}/\norm{m_t}) \mid \mathcal{F}_t ] \\
    \geq 
    \E[\min\{ \log(\norm{m_{t+1}}/\norm{m_t}), 0 \} \mid \mathcal{F}_t ] \geq - 1/d
\end{multline}
for an arbitrary measurable objective function $h$ with $d \geq 2$. 
This shows condition C2 in \Cref{prop:convergence_rate} with $B^\mathrm{lower} = 1/d$. Hence, it suffices to show condition C3.

Under \Cref{asm:L_U_2diff}, we have $\log(2/U) \leq \log(\norm{x}^2) - \log(f(x)) \leq \log(2/L)$. Therefore,
\begin{multline}
    \frac{1}{2}\log\left(\frac{L}{U}\right)
    \leq
    \log\left(\frac{\norm{m_{t+1}}}{\norm{m_t}}\right)
    - \frac{1}{2} \log\left(\frac{f(m_{t+1})}{f(m_t)}\right)
    \\
    \leq
    \frac{1}{2}\log\left(\frac{U}{L}\right).
\end{multline}
\Cref{lemma:variancebound} along with {$x^2 \leq 2 (\exp\abs{x} - 1)$} and the above inequality implies
\begin{equation}
\begin{aligned}[t]
&{\color{blue}\Var\left[\log\left(\frac{\norm{m_{t+1}}}{\norm{m_t}}\right)\right]}
\leq
{\color{blue}\E\left[\left(\log\left(\frac{\norm{m_{t+1}}}{\norm{m_t}}\right)\right)^2\right] }
\\
&\leq \frac{1}{2}{\color{blue}\E\left[\left(\log\left(\frac{f(m_{t+1})}{f(m_t)}\right)\right)^2\right]} + \frac{1}{2}\left(\log\left(\frac{U}{L}\right)\right)^2
\\
&\leq
\left(\frac{U}{L}\cdot\left(1+\frac{1}{d-3}\right) -1\right) + \frac{1}{2}\left(\log\left(\frac{U}{L}\right)\right)^2
<\infty
.
\end{aligned}
\end{equation}
This shows condition C3 in \Cref{prop:convergence_rate}.
This completes the proof.
\end{proof}

\begin{figure}[t]
    \centering
    \begin{subfigure}{0.33\hsize}%
    \includegraphics[width=\hsize, trim=10 15 10 10, clip]{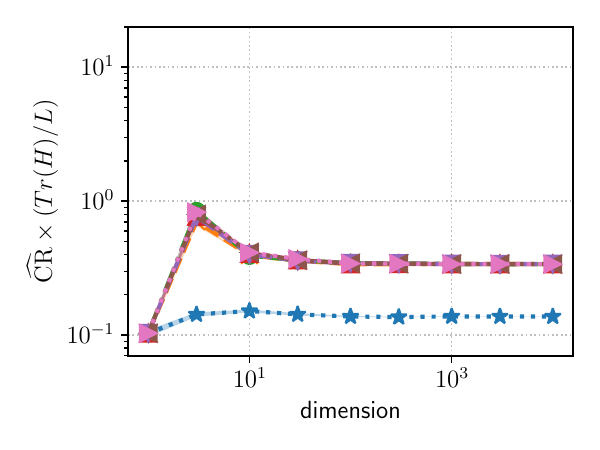}%
    \\
    \includegraphics[width=\hsize, trim=10 15 10 10, clip]{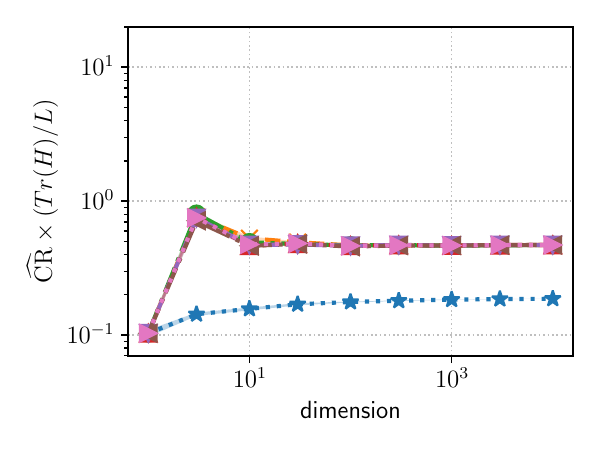}%
    \\
    \includegraphics[width=\hsize, trim=10 15 10 10, clip]{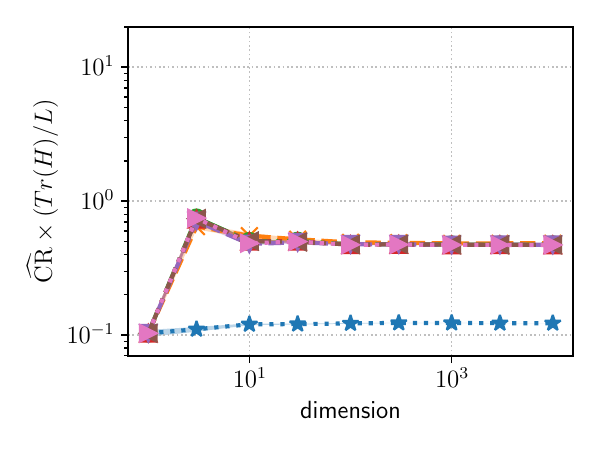}%
    \caption{$H_1$}\label{fig:cigar}%
    \end{subfigure}%
    \begin{subfigure}{0.33\hsize}%
    \includegraphics[width=\hsize, trim=10 15 10 10, clip]{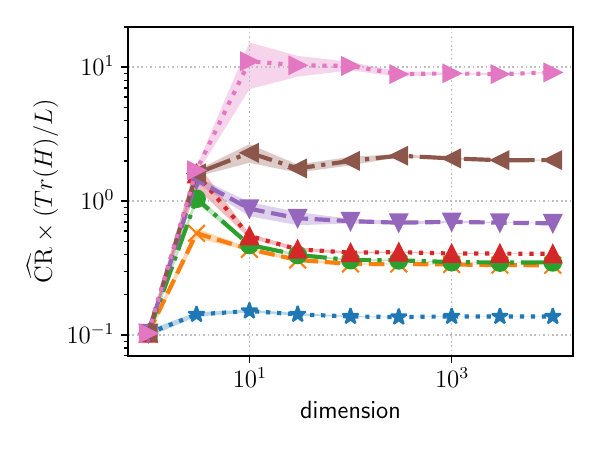}%
    \\
    \includegraphics[width=\hsize, trim=10 15 10 10, clip]{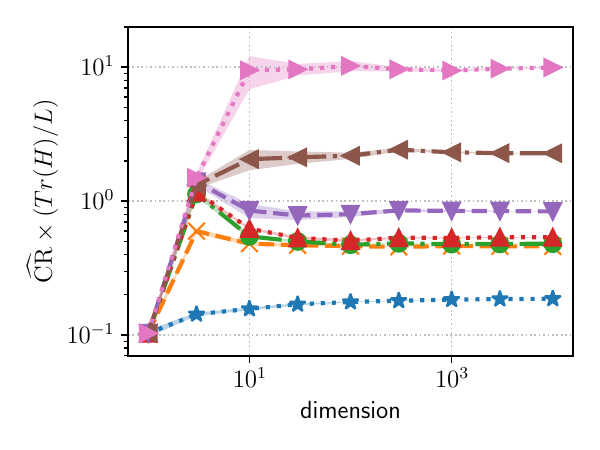}%
    \\
    \includegraphics[width=\hsize, trim=10 15 10 10, clip]{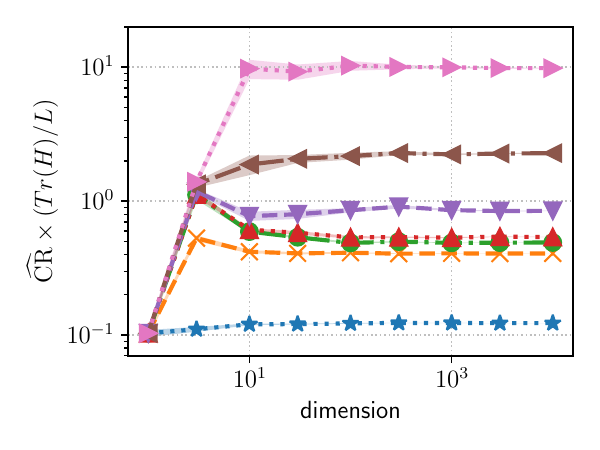}%
    \caption{$H_2$}\label{fig:ellipsoid}%
    \end{subfigure}%
    \begin{subfigure}{0.33\hsize}%
    \includegraphics[width=\hsize, trim=10 15 10 10, clip]{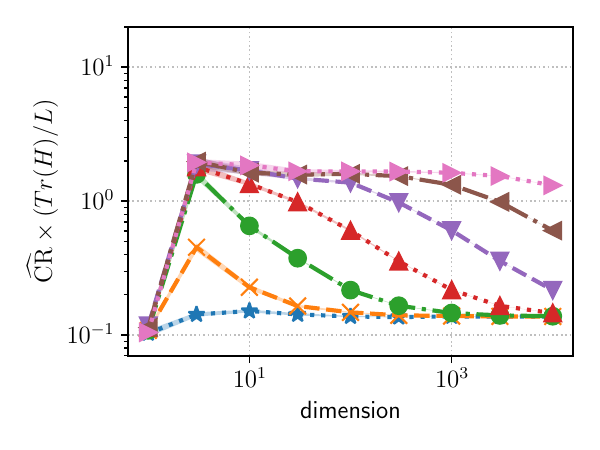}%
    \\
    \includegraphics[width=\hsize, trim=10 15 10 10, clip]{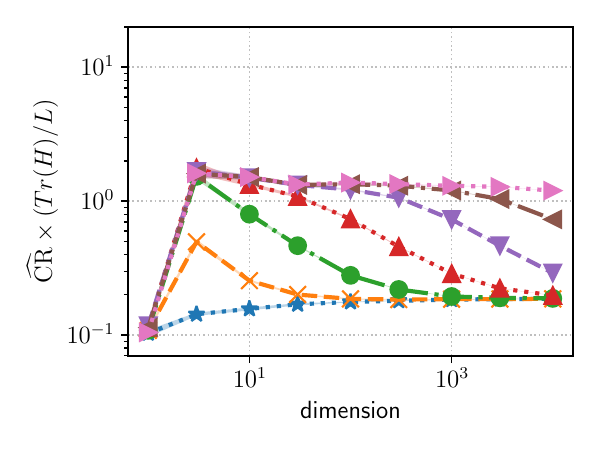}%
    \\
    \includegraphics[width=\hsize, trim=10 15 10 10, clip]{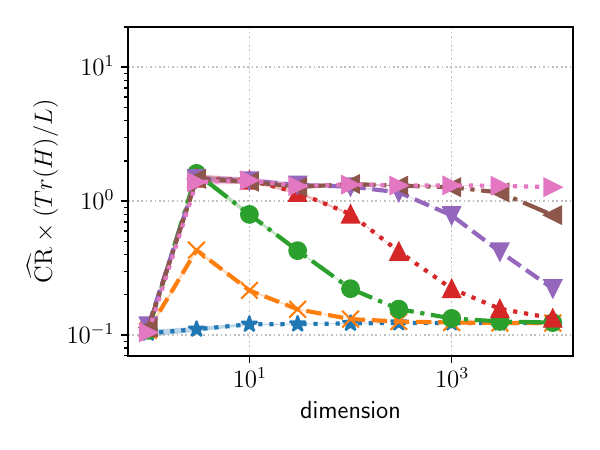}%
    \caption{$H_3$}\label{fig:discus}%
    \end{subfigure}%
    \caption{Results on convex quadratic functions with Hessian matrix $H_1$, $H_2$, and $H_3$. Average and standard error of $\widehat{\mathrm{CR}} \times (\Tr(H) / L)$ with varying $d$ and varying $\kappa$: ($\star$) $0$, ($\times$) $1$, ($\circ$) $2$, ($\triangle$) $3$, ($\triangledown$) $4$, ($\lhd$) $5$, ($\rhd$) $6$. Top: $\aup = \exp(1)$. Middle: $\aup = \exp(1/\sqrt{d})$. Bottom: $\aup = \exp(1/d)$.}
    \label{fig:all}
\end{figure}

\begin{figure}[t]
    \centering
    \begin{subfigure}{0.33\hsize}%
    \includegraphics[width=\hsize, trim=10 15 10 10, clip]{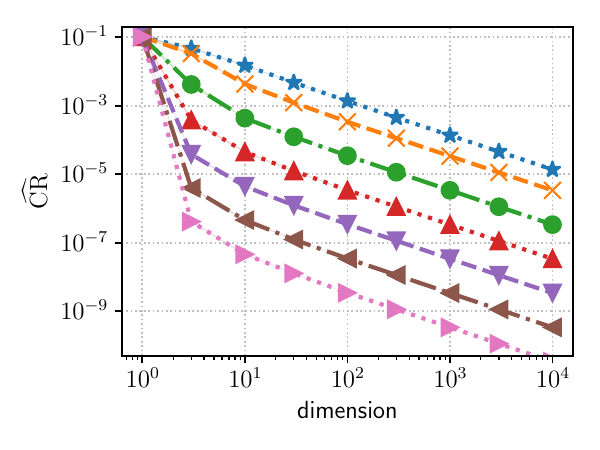}%
    \\
    \includegraphics[width=\hsize, trim=10 15 10 10, clip]{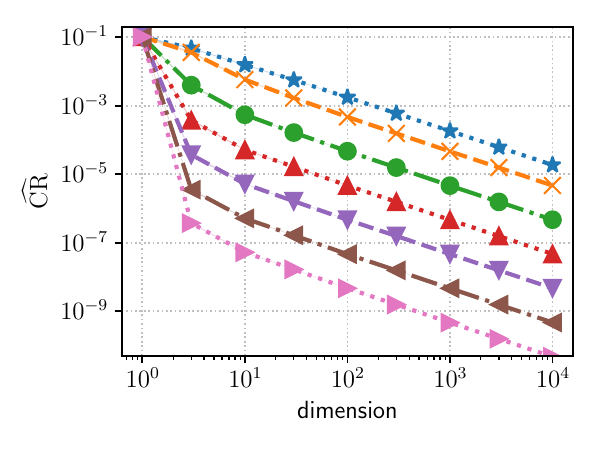}%
    \\
    \includegraphics[width=\hsize, trim=10 15 10 10, clip]{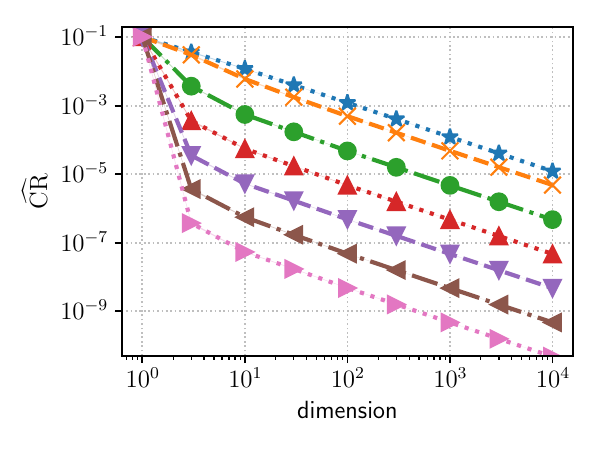}%
    \caption{$H_1$}\label{fig:cigar_woscale}%
    \end{subfigure}%
    \begin{subfigure}{0.33\hsize}%
    \includegraphics[width=\hsize, trim=10 15 10 10, clip]{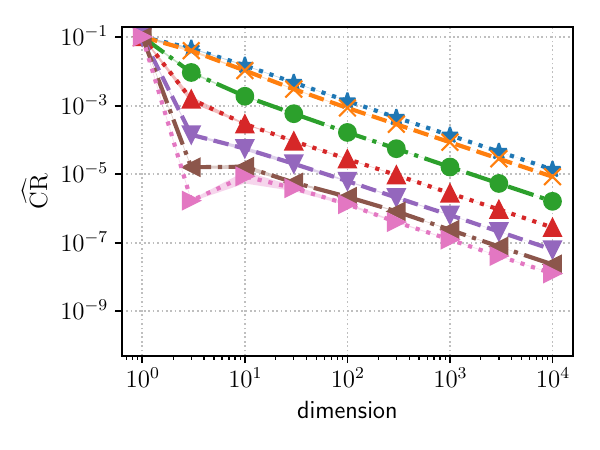}%
    \\
    \includegraphics[width=\hsize, trim=10 15 10 10, clip]{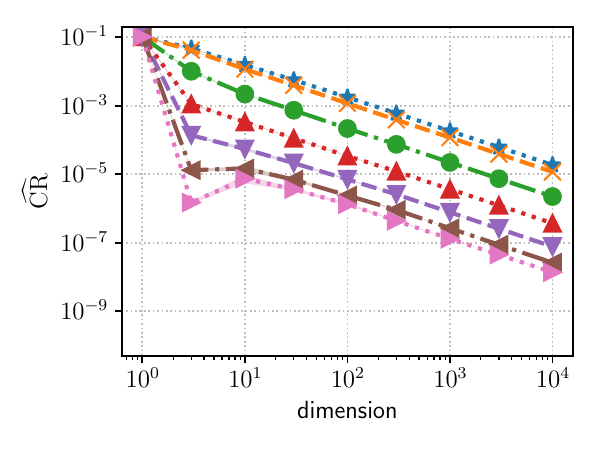}%
    \\
    \includegraphics[width=\hsize, trim=10 15 10 10, clip]{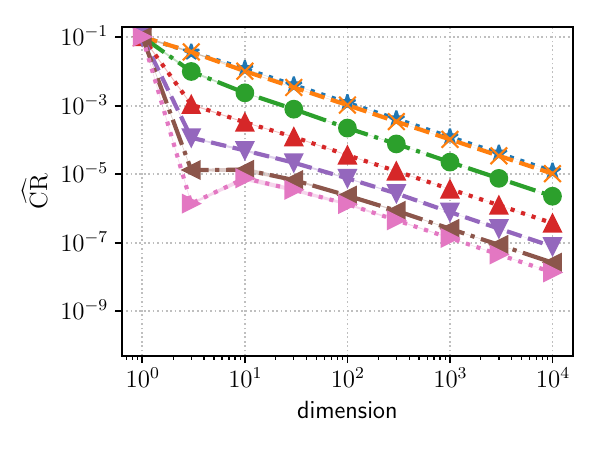}%
    \caption{$H_2$}\label{fig:ellipsoid_woscale}%
    \end{subfigure}%
    \begin{subfigure}{0.33\hsize}%
    \includegraphics[width=\hsize, trim=10 15 10 10, clip]{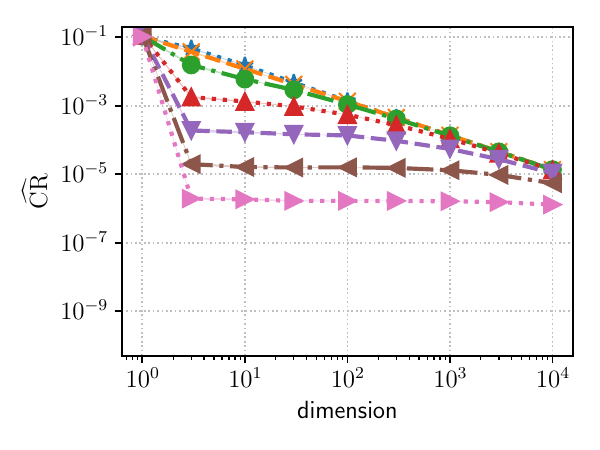}%
    \\
    \includegraphics[width=\hsize, trim=10 15 10 10, clip]{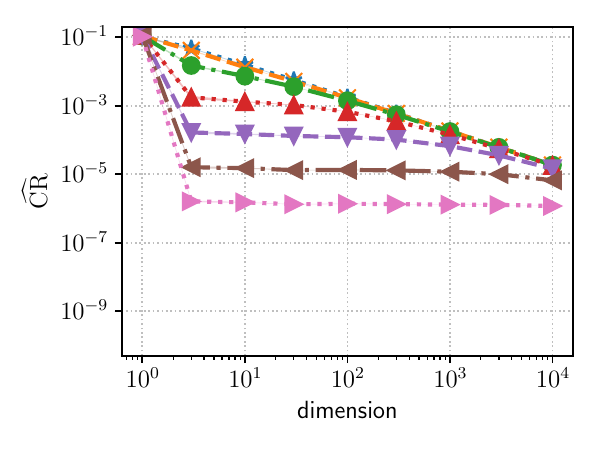}%
    \\
    \includegraphics[width=\hsize, trim=10 15 10 10, clip]{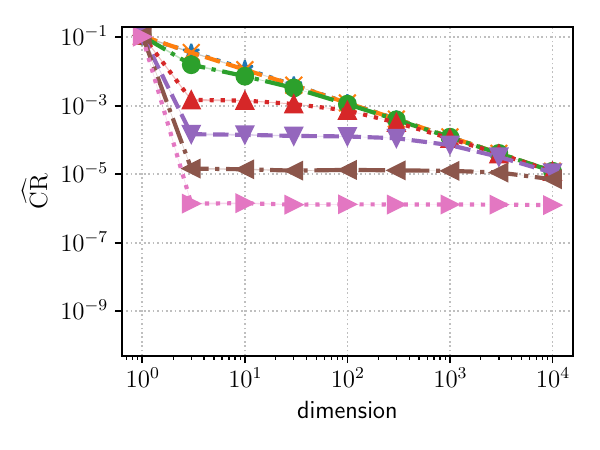}%
    \caption{$H_3$}\label{fig:discus_woscale}%
    \end{subfigure}%
    \caption{Results on convex quadratic functions with Hessian matrix $H_1$, $H_2$, and $H_3$. Average and standard error of $\widehat{\mathrm{CR}}$ with varying $d$ and varying $\kappa$: ($\star$) $0$, ($\times$) $1$, ($\circ$) $2$, ($\triangle$) $3$, ($\triangledown$) $4$, ($\lhd$) $5$, ($\rhd$) $6$. Top: $\aup = \exp(1)$. Middle: $\aup = \exp(1/\sqrt{d})$. Bottom: $\aup = \exp(1/d)$.}
    \label{fig:all_woscale}
\end{figure}

\subsection{Numerical Evaluation}\label{sec:exp}

We conducted numerical experiments to estimate the convergence rate of the (1+1)-ES with success-based step-size adaptation on different convex quadratic functions with different hyper parameter value $\aup$. We set $\adown = \aup^{-1/4}$ and $\aup = \exp(1)$, $\exp(1/\sqrt{d})$ and $\exp(1/d)$. 

We consider three convex quadratic functions $f(x) = \frac{1}{2} x^\mathrm{T} H x$ with Hessian matrices defined as 
\begin{align}
    H_1 &= \diag\left(1, 10^{\kappa}, \dots, 10^{\kappa}\right),
    \\
    H_2 &= \diag\left(10^{\kappa\cdot\frac{0}{d-1}}, 10^{\kappa\cdot\frac{1}{d-1}}, \dots, 10^{\kappa\cdot\frac{d-1}{d-1}}\right),
    \\
    H_3 &= \diag\left(1, \dots, 1, 10^{\kappa}\right).
\end{align}
In all cases, the greatest and smallest eigenvalues of the Hessian matrices are $U = 10^{\kappa}$ and $L = 1$, while the trace of the Hessian is $\Tr(H) = 1 + (d-1)10^{\kappa}$, $\sum_{i=1}^{d} 10^{\kappa \cdot \frac{i-1}{d-1}}$, and $d-1 + 10^{\kappa}$ for $H_1$, $H_2$ and $H_3$, respectively. Note that they all reduce to $H_0 = I$ when $\kappa = 0$.

The convergence rate is estimated by running the (1+1)-ES for $T = 10000 + 1000 \cdot d$ iterations. We estimate $\mathrm{CR}$ in \Cref{def:convergence} as $\widehat{\mathrm{CR}}$ by the least square linear regression of $\log(m_t)$ over $t \in \llbracket 0.9 T +1, T\rrbracket$. 
To avoid the numerical error, if $f(m_{\tilde{t}}) < 10^{-100}$ is reached at iteration $\tilde{t}$, we stop the run and set $T = \tilde{t}$ when estimating the convergence rate.
The initial solution $m_0$ is drawn from $\mathcal{N}(0, I)$ for each trial and the initial step-size is set to $\sigma_0 = \norm{\nabla f(m_0)} / \Tr(H)$. 
We run 10 independent trials.

\Cref{fig:all,fig:all_woscale} show the average and the standard deviation of $\widehat{\mathrm{CR}}$ on functions with $H_1$, $H_2$, and $H_3$, respectively, with condition number $10^{\kappa}$ for $\kappa = 0, \dots, 6$ on dimension $d = 1, 3, 10, 30, 100, 300, 1000, 3000, 10000$. 
Particularly, \Cref{cor:upper} implies $\mathrm{CR} \in \Omega(L / Tr(H))$ as $d \to \infty$ and \Cref{thm:lowerconvergencerate} implies $\mathrm{CR} \leq 1/d$. 

First, we observe that $\widehat{\mathrm{CR}}$ values observed for different $\aup$ values are very similar. 
Although it takes $d$ times more iterations to increase or decrease the step-size for the same amount when $\aup = \exp(1/d)$ than when $\aup = \exp(1)$, the difference in $\widehat{\mathrm{CR}}$ was less than the factor of $2$. However, as \Cref{theorem:lower} states, $\mathrm{CR}$ will be decreased as $d$ increases if we set $\log(\aup) \in o(1/d)$.  

Second, we observe the greatest $\widehat{\mathrm{CR}}$ when $\kappa = 0$ for all Hessian settings. We observe $\widehat{\mathrm{CR}} \approx 0.1 / d$. Therefore, the result of \Cref{thm:lowerconvergencerate} is tight up to a constant factor of approximately $0.1$. 

Third, we observe that $\widehat{\mathrm{CR}} \times (\Tr(H) / L) \geq 0.1$ for all cases.
This implies the implication of \Cref{cor:upper} (i.e., $\mathrm{CR} \in \Omega(L/\Tr(H))$) is valid for small $d$ although it is derived for the limit of $d$ to infinity.

Fourth, we observe that $\widehat{\mathrm{CR}} \times (\Tr(H) / L) \in [0.1, 2.0]$ on $H_1$ and $H_3$ for all $\kappa$ and for all $d$. This implies $\mathrm{CR} \in \Omega(L/\Tr(H))$ is a tight bound up to a constant factor on $H_1$ and $H_3$. 
However, we observe that $\widehat{\mathrm{CR}} \times (\Tr(H) / L) \in [0.1, 2.0]$ on $H_2$ for $\kappa \leq 3$ and for all $d$, but $\widehat{\mathrm{CR}} \times (\Tr(H) / L)$ becomes greater as $\kappa$ increases for $\kappa \geq 4$. We conjecture that it is because the (1+1)-ES did not reach the stationary regime within the given budget on $H_2$ with $\kappa \geq 4$ and $\mathrm{CR}$ was overestimated. Thus, we observed a significantly greater $\widehat{\mathrm{CR}}$ if we reduce $T$ by the factor of $10$.

Fifth, we observe that $\widehat{\mathrm{CR}}$ scales differently with respect to $d$ on $H_1$ and $H_3$. If we consider the naive bound derived in \Cref{cor:upper} (i.e., $\mathrm{CR} \in \Omega(L/(d\cdot U)))$), the bounds for these two problems are the same. However, if we consider the bound specialized for convex quadratic function derived in \Cref{cor:upper} (i.e., $\mathrm{CR} \in \Omega(L/\Tr(H))$), we find that the bounds are $1 / (1 + (d-1)10^{\kappa})$ for $H_1$ and $1 / (d-1 + 10^{\kappa})$ for $H_3$. The bound for $H_1$ is well approximated by the naive bound. The bound for $H_3$ is approximated by $1/10^{\kappa}$ if $d < 10^{\kappa}$ and $1/d$ if $d > 10^{\kappa}$. This clearly appears in \Cref{fig:discus_woscale}. Namely, the convergence rate of the (1+1)-ES does not always rely on the ratio $L/(d\cdot U)$ as in the naive bound. However, it can be as good as $1/d$ when the number of axes sensitive to the objective function (corresponding to the eigen vectors for large eigenvalues of the Hessian matrix) is limited.

\section{Discussion}

We derived the convergence rate bounds for the (1+1)-ES with success-based step-size adaptation on functions that are Lipschitz smooth and strongly convex around the global optimum and their strictly increasing transformation (i.e., $\mathcal{P}_{d, L, U}^{\gamma}$ defined in \Cref{def:problem}). To the best of our knowledge, this is the first study that shows the convergence rate dependency of the form $\frac{L}{d U}$ on dimension $d$, Lipschitz smoothness parameter $U$, and strong convexity parameter $L$ explicitly on $\mathcal{P}_{d, L, U}^{\gamma}$. The tighter convergence rate bounds of form $\frac{L}{\Tr(H)}$ on convex quadratic functions $\frac12 x^\T H x$ derived in this study provide further insight into the algorithmic behavior. In particular, the fact that the convergence rate can be independent of the condition number $\Cond(H) {\leq} U / L$ on functions where a few eigenvalues of $H$ are large and the others are all relatively small is a promising characteristic of the (1+1)-ES when applied to high dimensional problems with a few important variables.

Although our results are stated as asymptotic ones because our focus is to derive the convergence rate, we can easily derive a non-asymptotic bound if we limit our attention to Lipschitz smooth and strongly convex functions (i.e., $\mathcal{S}_{d,L,U}$). Noting that our potential function $V(\theta)$ satisfies $\log(f(m)) \leq V(\theta)$, we have
\begin{equation}
\begin{split}
    \MoveEqLeft[0]\E\left[\log\left( f(m_t) \right)\right] - \log\left(f(m_0)\right)\\
    &\leq \E\left[\log\left( V(\theta_t) \right)\right] - \log\left( V(\theta_0) \right) + \log\left( V(\theta_0) \right) - \log\left( f(m_0) \right)
    \\
    &\leq - 2\cdot B_{d, L, U}^\mathrm{upper} \cdot t +  \log\left( V(\theta_0) \right) - \log\left( f(m_0) \right) .
\end{split}
\end{equation}

Three strong points of the (1+1)-ES with success-based step-size adaptation over other derivative-free approaches analyzed on $\mathcal{S}_{d,L,U}$ are as follows. 
First, any hyper-parameters need not be set depending on some characteristic parameters such as $L$ and $U$.
Without the use of these parameters inside the algorithm, the convergence rate factor of $\frac{L}{d U}$ is achieved. 
As estimating $L$ and $U$ in simulation-based optimization settings is not trivial, it is a strong advantage of the (1+1)-ES.
Second, the (1+1)-ES is invariant to any strictly increasing transformation of the objective function; hence, our analysis is not limited to $\mathcal{S}_{d,L,U}$, but covers a broader class of functions, namely $\mathcal{P}_{d, L, U}^{\gamma}$.
Third, our results guarantee linear convergence of the (1+1)-ES, rather than locating an $\epsilon$-optimal solution such that $f(x) - f(x^*) \leq \epsilon$ for a finite $\epsilon$. 
Previous studies on derivative-free approaches \cite{2019gradientlesgolovins, nesterov2017random} set some hyper-parameters of the algorithms depending on this pre-defined threshold $\epsilon$ to derive these results.

We close our analysis with discussion on two limitations of this study. 
First, we have derived the convergence rate bound in \Cref{theorem:lower}.
However, the dependency of $B_{d, L, U}^\mathrm{upper}$ on $d$, $L$, and $U$ are not explicitly derived for finite $d$. 
Second, to evaluate $B_{d, L, U}^\mathrm{upper}$ in the limit $d \to \infty$, we have assumed that $\aup$ and $\adown$ are set such that $\Phi\left(-\sqrt{\frac{2}{\pi}}\right) < p^\mathrm{target} < \frac{1}{2}$, where $p^\mathrm{target} = \frac{\log(1/\adown)}{\log(\aup / \adown)}$. 
However, as $\Phi\left(-\sqrt{\frac{2}{\pi}}\right)\approx 0.212$, the most common setting of $\aup$ and $\adown$ is such that $p^\mathrm{target} = 1/5$ \cite{rechenberg1973evolution} and is not included in the above assumption.
These limitations will be addressed in future works.

\section*{Acknowledgments}
This paper is partially supported by JSPS KAKENHI Grant Number 19H04179.

\appendix

\section{Proof}

\subsection{Proof of \Cref{prop:invariance}}\label{apdx:prop:invariance}

For the first claim, it is sufficient to show that $\mathcal{G}(\theta, z; f) = \mathcal{G}(\theta, z; g \circ f)$ for all $\theta \in \Theta$ and $z \in \R^d$. This is trivial because $f(m + \sigma \cdot z) \leq f(m) \Leftrightarrow g(f(m + \sigma \cdot z)) \leq g(f(m))$.

For the second claim, we first show that $\mathcal{G}((m, \sigma), z; f) = \mathcal{G}((m+x^*, \sigma), z; f \circ T)$ for all $\theta \in \Theta$ and $z \in \R^d$. Because $f(T(m + x^* + \sigma \cdot z)) \leq f(T(m + x^*)) \Leftrightarrow f(m + \sigma \cdot z) \leq f(m)$, it is clear that $\mathcal{G}((m, \sigma), z; f) = \mathcal{G}((m+x^*, \sigma), z; f \circ T)$. Assume that $\tilde{m}_{t} = m_t + x^*$ and $\tilde{\sigma}_t = \sigma_t$. Then, we have $\mathcal{G}(\theta_t, z; f) = \mathcal{G}(\tilde{\theta}_t, z; f \circ T)$, and hence $\tilde{m}_{t+1} = m_{t+1} + x^*$ and $\tilde{\sigma}_{t+1} = \sigma_{t+1}$. Because the assumption holds for $t = 0$, by mathematical induction, we obtain the second claim.

For the third claim, it is sufficient to show that $\mathcal{G}(\theta, z; f) = \mathcal{G}(\theta, z; \tilde{f})$ for all $\theta \in \Theta$ and $z \in \R^d$. Because the (1+1)-ES guarantees the monotone improvement of $f(m_t)$, we have $m_t \in S(m_0)$. Therefore, $\tilde{f}(m_t) = f(m_t)$ for all $t \geq 0$. We have $f(m_t + \sigma_t \cdot z_t) \leq f(m_t) \Rightarrow \tilde{f}(m_t + \sigma_t \cdot z_t) \leq \tilde{f}(m_t)$. Moreover, because $\tilde{f}(x) > \tilde{f}(m_0)$ for all $x \in \R^d\setminus S(m_0)$, we have $\tilde{f}(m_t + \sigma_t \cdot z_t) \leq \tilde{f}(m_t) \Rightarrow m_t + \sigma_t \cdot z_t \in S(m_0)$, implying $f(m_t + \sigma_t \cdot z_t) \leq f(m_t)$. Hence, $f(m_t + \sigma_t \cdot z_t) \leq f(m_t) \Leftrightarrow \tilde{f}(m_t + \sigma_t \cdot z_t) \leq \tilde{f}(m_t)$ and we obtain $\mathcal{G}(\theta, z; f) = \mathcal{G}(\theta, z; \tilde{f})$.

\subsection{Proof of \Cref{prop:convergence_rate}}\label{apdx:prop:convergence_rate}
Let $Z_{t+1} = X_{t+1} - \E[X_{t+1} \mid \mathcal{F}_t]$.
Then, $\{Z_{t}\}_{t \geq 1}$ is a martingale difference sequence adapted to $\{\mathcal{F}_{t}\}$, and
\begin{multline}
\frac1t (X_t - X_0)
= \frac1t \sum_{i=1}^{t}(X_i - X_{i-1})
= \frac1t \sum_{i=1}^{t}(Z_i + \E[X_i - X_{i-1} \mid \mathcal{F}_{i-1}])
\\
= \frac1t \sum_{i=1}^{t}Z_i + \frac1t \sum_{i=1}^{t} \E[X_i - X_{i-1} \mid \mathcal{F}_{i-1}]
.
\end{multline}
Suppose that C3 holds.
{\color{blue}Beucause $\Var[X_t] = \Var[Z_t] + \Var[\E[X_t\mid\mathcal{F}_{t-1}]] \geq \Var[Z_t]$, }
C3 immediately implies that $\sum_{t=1}^{\infty}{\color{blue}\E[Z_{t}^2]} / t^2 < \infty$. Consequently, from the strong law of large numbers of martingale \cite{chow1967strong}, we obtain $\lim_{t\to\infty} \frac1t \sum_{i=1}^{t} Z_t = 0$ almost surely. 
We obtain statement (I) by taking $\limsup$ of the equation above and using C1 and $\limsup_{t\to\infty} \frac1t \sum_{i=1}^{t} Z_t = 0$.
Similarly, we obtain statement (II) by taking $\liminf$ and using C2 and $\liminf_{t\to\infty} \frac1t \sum_{i=1}^{t} Z_t = 0$. This completes the proof.

\subsection{Proof of \Cref{prop:f-norm}}\label{apdx:prop:f-norm}
Let $T: x \mapsto x - x^\text{opt}$. Then, $T^{-1}(x) = x + x^\text{opt}$ and $S_T$ in (2) of \Cref{prop:invariance} is equal to $S$ in \Cref{prop:f-norm}.
Hence, in light of (1) and (2) of \Cref{prop:invariance}, we have $\tilde{\theta}_t = S(\theta_t)$ for all $t \geq 0$.
This implies $\norm{\tilde{m}_t - x^\mathrm{opt}} = \norm{m_t}$ for all $t \geq 0$. Hence, we obtain \eqref{eq:liminf}. 
Moreover, under \Cref{asm:L_U_2diff}, for any $x \in \R^d$, we have
$\frac{2 f(x)}{ U } \leq \norm{x}^2 \leq \frac{2 f(x)}{ L }$.
Then, we obtain
$\frac{1}{2}\log \left(\frac{L}{U}\right)
\leq \log\left(\frac{ \norm{m_t}} { \norm{m_0} }\right) - \frac{1}{2}\log\left(\frac{f(m_t)}{f(m_0)}\right)
\leq \frac{1}{2}\log \left(\frac{U}{L}\right)$.
Taking $\limsup$ after multiplying all terms by $\frac{1}{t}$, we obtain \eqref{eq:limsup}.

\subsection{Proof of \Cref{lemma:variance}}\label{apdx:lemma:variance}


We prove the statements one by one.

1. The $L$-strong convexity and $U$-Lipschitz smoothness of $f$ implies 
$L \norm{z}^2 \leq \mathcal{Q}_z(\theta) \leq U \norm{z}^2$.
This inequality and $\E[\norm{z}^2] = d$ complete the proof for the first statement.

2. \cite[Therem2]{chen1982inequality}
shows that
$\Var[\mathcal{Q}_z(\theta)] \leq \E[ \norm{\nabla_z \mathcal{Q}_z(\theta)}^2 ]$.
Using this fact and the $U$-Lipschitz continuity of $\nabla f$, we have
\begin{equation}
\begin{split}
\Var[\mathcal{Q}_z(\theta)] 
&\leq \E[ \norm{\nabla_z \mathcal{Q}_z(\theta)}^2 ] \\
&= 
(4 / \sigma^2) 
\E[ \norm{ \nabla f(m + \sigma z) - \nabla f(m) }^2 ] 
\\
&\leq 
(4 / \sigma^2) 
\E[ U^2 \sigma^2 \norm{ z }^2 ] 
= 4 d U^2 .
\end{split}
\end{equation}

3. We use the Cauchy-Schwarz inequality and obtain
\begin{equation}
\begin{split}
\MoveEqLeft[1]
\E[\mathcal{Q}_z(\theta) \cdot \ind{z_e \leq 0}] \\
&= \E[\E[\mathcal{Q}_z(\theta)] \cdot \ind{z_e \leq 0}] + \E[(\mathcal{Q}_z(\theta) - \E[\mathcal{Q}_z(\theta)]) \cdot \ind{z_e \leq 0}]
\\
&= \E[\mathcal{Q}_z(\theta)] / 2 + \E[(\mathcal{Q}_z(\theta) - \E[\mathcal{Q}_z(\theta)]) \cdot \ind{z_e \leq 0}]
\\
&\leq \E[\mathcal{Q}_z(\theta)] / 2 + (\E[(\mathcal{Q}_z(\theta) - \E[\mathcal{Q}_z(\theta)])^2] \cdot \E[\ind{z_e \leq 0}])^{1/2}
\\
&= \E[\mathcal{Q}_z(\theta)] / 2 + (\Var[\mathcal{Q}_z(\theta)] / 2)^{1/2} ,
\end{split}
\end{equation}
and similarly,
\begin{equation}
\begin{split}
\MoveEqLeft[1]
\E[\mathcal{Q}_z(\theta) \cdot \ind{z_e \leq 0}] \\
&= \E[\mathcal{Q}_z(\theta)] / 2 + \E[(\mathcal{Q}_z(\theta) - \E[\mathcal{Q}_z(\theta)]) \cdot \ind{z_e \leq 0}]
\\
&= \E[\mathcal{Q}_z(\theta)] / 2 - \E[(\mathcal{Q}_z(\theta) - \E[\mathcal{Q}_z(\theta)]) \cdot \ind{z_e > 0}]
\\
&\geq \E[\mathcal{Q}_z(\theta)] / 2 - (\E[(\mathcal{Q}_z(\theta) - \E[\mathcal{Q}_z(\theta)])^2] \cdot \E[\ind{z_e > 0}])^{1/2}
\\
&= \E[\mathcal{Q}_z(\theta)] / 2 - (\Var[\mathcal{Q}_z(\theta)] / 2)^{1/2} .
\end{split}
\end{equation}
Note that we have $\E[\mathcal{Q}_z(\theta)] \geq d L$. 
We have,
\begin{equation}
    \begin{split}
(\Var[\mathcal{Q}_z(\theta)]/ 2)^{1/2} 
&\leq (2 d)^{1/2} U 
\\
&= \E[\mathcal{Q}_z(\theta)] ((2 d)^{1/2} U / \E[\mathcal{Q}_z(\theta)] )
\\
&\leq \E[\mathcal{Q}_z(\theta)] ((2 / d)^{1/2} (U / L) ) .
\end{split}
\end{equation}
This ends the proof.

\subsection{Proof of \Cref{cor:tr_cond}}\label{apdx:cor:tr_cond}
In light of \Cref{lemma:variance}, we have for any $\theta \in \Thetazero$
\begin{equation}
\abs*{
\frac{\E[\mathcal{Q}_z(\theta)\cdot\ind{z_e\leq 0}]}{\E[\mathcal{Q}_z(\theta)]} - \frac{1}{2}
}
\leq (2 / d)^{1/2} (U/L) \xrightarrow{d \to \infty} 0 .
\end{equation}
It implies that for any $\epsilon > 0$ there exists $D' \in \N$ such that $\kappa_{\inf} \in (2 - \epsilon, 2+\epsilon)$ for all $d \geq D'$. 
Then, the RHS of \eqref{eq:tr_cond} is lower bounded by a positive constant, say $c > 0$, for all $d \geq D'$. 
Again in light of \Cref{lemma:variance}, we have
\begin{equation}
    \mathcal{V}_\mathrm{std}
    \leq 4 d U^2/ (d L)^2 = 4 U^2/ (d L^2 ) \xrightarrow{d \to \infty} 0.
\end{equation}
It implies that there exists $D'' \in \N$ such that the LHS of \eqref{eq:tr_cond} is strictly smaller than the above defined $c > 0$ for all $d \geq D''$. 
Therefore, letting $D = \max\{D', D''\}$, we obtain \eqref{eq:tr_cond} for all $d \geq D$. This completes the proof.

\subsection{Proof of \Cref{lemma:qualitygain}}\label{apdx:lemma:qualitygain}

Let $G_a(z) := (f(m + \sigma z) - f(m) ) \ind{ z_e \leq 0}$ and $G_b(z) := \ind{f(m + \sigma z) \leq f(m)}$. 
First, we prove that the LHS of \eqref{eq:qualitygain} is upper-bounded by $\E[G_a(z)]\E[G_b(z)] / f(m)$.
Because $f: \R^d \to \R$ is a differentiable convex function, we have the following property: 
(A) $\ind{f(m + \sigma z) \leq f(m)} \leq \ind{ z_e \leq 0}$, implying that $\ind{f(m + \sigma z) \leq f(m)} = \ind{ z_e \leq 0}\ind{f(m + \sigma z) \leq f(m)}$;
(B) $G_a(z)$ and $G_b(z)$ are negatively correlated, that is, $(G_a(z_1) - G_a(z_2))(G_b(z_1) - G_b(z_2)) \leq 0$ for all $z_1, z_2 \in \R^d$.
Because $((f(m + \sigma z) - f(m)) / f(m)) \ind{f(m + \sigma z) \leq f(m)} = G_a(z) G_b(z) / f(m)$, where we used (A), the LHS of \eqref{eq:qualitygain} is  $\E[G_a(z) G_b(z)] / f(m)$.
Using (B) along with the chebyshev sum inequality \cite[Theorem 236]{hardy1952inequalities}, we have $\E[G_a(z)G_b(z)] \leq \E[G_a(z)]\E[G_b(z)]$. Hence, the LHS of \eqref{eq:qualitygain} is upper-bounded by $\E[G_a(z)]\E[G_b(z)] / f(m)$.

Next, we prove that the RHS of \eqref{eq:qualitygain} is $\E[G_a(z)]\E[G_b(z)] / f(m)$. 
First, we have $\E[G_b(z)] = \E[\ind{f(m + \sigma z) \leq f(m)}] = \Pr[f(m + \sigma z) \leq f(m)]$.
Second,
\begin{equation}
\begin{aligned}[t]
&\E[G_a(z)]
\\
=&
  \E\left[\left(\sigma \| \nabla f(m) \| z_e + \frac{\sigma^2 }{2} \mathcal{Q}_z(\theta) \right) \ind{z_e \leq 0}\right]
\\
=& - \frac{\sigma \| \nabla f(m) \|}{\sqrt{2\pi}} + \frac{\sigma^2}{2}\cdot\E\left[\mathcal{Q}_z(\theta) \cdot \ind{z_e \leq 0}\right]
\\
=& \sigma \norm{\nabla f(m)}\cdot\left(
- \frac{1}{\sqrt{2\pi}}
+ \frac{\sigma}{2\norm{\nabla f(m)}}\cdot\E\left[\mathcal{Q}_z(\theta) \cdot \ind{z_e \leq 0}\right]
\right)
  .
\end{aligned}
\end{equation}
These equalities demonstrate that the RHS of \eqref{eq:qualitygain} is $\E[G_a(z)]\E[G_b(z)] / f(m)$.

\subsection{Proof of \Cref{lemma:variancebound}}\label{apdx:lemma:variancebound}

Under \Cref{asm:L_U_2diff}, for any $x \in \R^d$, we have
$\frac{2 f(x)}{ U } \leq \norm{x}^2 \leq \frac{2 f(x)}{ L }$.
Then, we have
\begin{equation}
\begin{aligned}
\MoveEqLeft[3]
\log\left(\frac{f(m+\sigma z)}{f(m)}\right)\cdot\ind{f(m+\sigma z)\leq f(m)}
\\
=&
\min\left\{\log\left(\frac{f(m+\sigma z)}{f(m)}\right), 0 \right\}
\\
\geq&
\min\left\{
\log\left(\frac{\norm{m+\sigma z}^2}{\norm{m}^2}\right) + \log\left(\frac{L}{U}\right)
, 0 \right\}
\\
=&
\min\Bigg\{
\log\left(
1 + \frac{\norm{z}^2}{\norm{m}^2}\left(
\left(\sigma + \frac{\langle m, z \rangle}{\norm{z}^2}\right)^2
-\left(\frac{\langle m, z \rangle}{\norm{z}^2}\right)^2
\right)
\right) 
\\
&+ \log\left(\frac{L}{U}\right)
, 0 \Bigg\}
\\
\geq&
\log\left(
1 
-\frac{\langle m, z \rangle^2}{\norm{m}^2\cdot\norm{z}^2}
\right)
+ \log\left(\frac{L}{U}\right)
.
\end{aligned}
\end{equation}
Therefore,
\begin{equation}
\begin{aligned}[t]
&\exp\left(
\abs*{
\log\left(\frac{f(m+\sigma z)}{f(m)}\right)\cdot\ind{f(m+\sigma z)\leq f(m)}
}
\right)
\\
&\leq
\exp\left(
-
\log\left(
\frac{L}{U}\cdot\left(1 
-\frac{\langle m, z \rangle^2}{\norm{m}^2\cdot\norm{z}^2}
\right)
\right)
\right)
\\
&=
\frac{U}{L}\cdot \left( 1 + \frac{1}{d-1}\cdot \frac{(d-1) \inner*{\frac{m}{\norm{m}}}{z}^2}{\norm{z}^2-\inner*{\frac{m}{\norm{m}}}{z}^2}\right)
.
\end{aligned}
\end{equation}
Here, $\frac{(d-1) \inner*{\frac{m}{\norm{m}}}{z}^2}{\norm{z}^2-\inner*{\frac{m}{\norm{m}}}{z}^2}$ is $F$-distributed and its expectation is $(d-1) / (d-3)$. 
Finally, we obtain \eqref{eq:lemma:variancebound}.
This completes the proof.

\subsection{Proof of \Cref{lemma:successprobability}}\label{apdx:lemma:successprobability}

Let 
$A = \bar\sigma/2$
and $B = \mathcal{Q}_z(\theta)/\E[\mathcal{Q}_z(\theta)]$.
First, we derive a few inequalities before deriving the upper bound and the lower bound.
By using Chebyshev's inequality, we obtain
\begin{align}
\Pr\left[\left| B - 1 \right|\geq \epsilon\right]
\leq
\frac{1}{\epsilon^2} \cdot \mathcal{V}_\mathrm{std}
.
\end{align}
Because of \eqref{eq:q-def}, we can write the success probability as
\begin{equation}
\begin{aligned}[t]
\Pr[ f(m + \sigma z) \leq f(m)]
= \Pr\left[ z_e \leq - A \cdot B\right]
.
\end{aligned}
\end{equation}

Now, we derive the upper bound of the success probability as
\begin{equation}
\begin{aligned}[t]
\Pr\left[f(m + \sigma z) \leq f(m) \right]
=&
\Pr\left[ \left(z_e \leq - A \cdot B\right)\wedge \left( B > 1-\epsilon\right)\right]
\\
&+\Pr\left[ \left(z_e \leq - A \cdot B\right)\wedge\left(B \leq 1-\epsilon\right)\right]
\\
\leq&
\Pr\left[ z_e \leq - A \cdot(1-\epsilon)\right]
+\Pr\left[ \left| B -1 \right| \geq \epsilon\right]
\\
\leq&
\Phi\left( - \frac12 \bar\sigma\cdot(1-\epsilon)\right)
+ \frac{1}{\epsilon^2} \cdot \mathcal{V}_\mathrm{std}
.
\end{aligned}
\end{equation}

Analogously, we derive the lower bound:
\begin{equation}
\begin{aligned}[t]
\Pr\left[f(m + \sigma z) > f(m) \right]
\leq&
\Pr\left[ z_e > - A \cdot(1+\epsilon)\right]
+\Pr\left[ \left| B -1 \right| \geq \epsilon\right]
\\
\leq&
1 - \Phi\left( - \frac12 \bar\sigma\cdot(1+\epsilon)\right)
+ \frac{1}{\epsilon^2} \cdot \mathcal{V}_\mathrm{std}
.
\end{aligned}
\end{equation}
This completes the proof.

\subsection{Proof of \Cref{cor:successprobability}}\label{apdx:cor:successprobability}


First, we prove \eqref{eq:cor:successprobability:large}.
Define 
\begin{equation}
B_\theta^\mathrm{high}(q; \epsilon) := 2\cdot\Phi^{-1} \left(1 - \left(q + \frac{1}{\epsilon^2} \cdot \mathcal{V}_\mathrm{std}\right) \right)/ (1 + \epsilon)
\label{eq:b_high_q_e}
\end{equation}
for $q\in(0, 1/2)$ and $\epsilon > \sqrt{\frac{2}{1-2q}\cdot\mathcal{V}_\mathrm{std}}$.
Apparently, $B_\theta^\mathrm{high}(q; \epsilon)$ is continuous with respect to both arguments and $B_\theta^\mathrm{high}(q; \epsilon) > 0$. 
Since $B_\theta^\mathrm{high}(q; \epsilon) \to 0$ as $\epsilon \downarrow \sqrt{\frac{2}{1-2q}\cdot\mathcal{V}_\mathrm{std}}$ and $B_\theta^\mathrm{high}(q; \epsilon) \to 0$ as $\epsilon \uparrow + \infty$, there must exist at least one $\epsilon$ for each $q$ such that $B_\theta^\mathrm{high}(q; \epsilon)$ is maximized. Let $\epsilon^\text{high}(q)$ be the smallest one of such maximizers. Then, $B_\theta^\text{high}(q) = B_\theta^\text{high}(q; \epsilon^\text{high}(q))$.

In light of \Cref{lemma:successprobability}, we have
\begin{equation}
\begin{aligned}[t]
\bar\sigma \leq  B_\theta^\mathrm{high}(q; \epsilon)
&\Leftrightarrow 
q\leq \Phi\left(-\frac12 \bar\sigma\cdot (1+\epsilon)\right) - \frac{1}{\epsilon^2}\cdot\mathcal{V}_\mathrm{std}
\\
&\Rightarrow
\Pr\left[f(m+\sigma z)\leq f(m)\right] > q
\end{aligned}
\end{equation}
for any $q\in(0, 1/2)$ and $\epsilon > \sqrt{\frac{2}{1-2q}\cdot\mathcal{V}_\mathrm{std}}$.
By letting $\epsilon = \epsilon^\text{high}(q)$, the above relation implies \eqref{eq:cor:successprobability:large}. This completes the proof of \eqref{eq:cor:successprobability:large}.
It is trivial to see that $B_\theta^\mathrm{high}(q) \leq 2 \Phi^{-1}(1-q)$ for all $q \in \left(0, \frac12\right)$.

Similarly, we prove \eqref{eq:cor:successprobability:small}.
Define
\begin{equation}
B_\theta^\mathrm{low}(q; \epsilon)
:= 
2 \cdot\Phi^{-1} \left(1 - \left(q - \frac{1}{\epsilon^2} \cdot \mathcal{V}_\mathrm{std}\right) \right)/ (1 - \epsilon)\label{eq:b_low_q_e}
\end{equation}
for $q\in\left( \mathcal{V}_\mathrm{std} , \frac{1}{2}\right)$ and $\epsilon\in \left(\sqrt{\frac{1}{q}\cdot\mathcal{V}_\mathrm{std}}, 1\right)$.
It is easy to see that $B_\theta^\mathrm{low}(q; \epsilon)$ is continuous with respect to both arguments and $B_\theta^\mathrm{low}(q; \epsilon) > 0$.
Moreover, since $B_\theta^\mathrm{low}(q; \epsilon) \to +\infty$ as $\epsilon \uparrow 1$ and $B_\theta^\mathrm{low}(q; \epsilon) \to +\infty$ as $\epsilon \downarrow \sqrt{\frac{1}{q}\cdot\mathcal{V}_\mathrm{std}}$, there must exist at least one $\epsilon$ for each $q$ such that $B_\theta^\mathrm{low}(q; \epsilon)$ is minimized. Let $\epsilon^\text{low}(q)$ be the smallest one of such minimizers. Then, $B_\theta^\text{low}(q) = B_\theta^\text{low}(q; \epsilon^\text{low}(q))$.

In light of \Cref{lemma:successprobability}, we have
\begin{equation}
\begin{aligned}[t]
    \label{eq:8}
\MoveEqLeft[2]
\bar\sigma \leq 
B_\theta^\mathrm{low}(q; \epsilon)
\Leftrightarrow
q\geq \Phi\left(-\frac12 \bar\sigma\cdot (1-\epsilon)\right) + \frac{1}{\epsilon^2}\cdot\mathcal{V}_\mathrm{std}
\\
&\Rightarrow
\Pr\left[f(m+\sigma z)\leq f(m)\right] < q
,
\end{aligned}
\end{equation}
for any $q\in\left( \mathcal{V}_\mathrm{std} , \frac{1}{2}\right)$ and $\epsilon\in \left(\sqrt{\frac{1}{q}\cdot\mathcal{V}_\mathrm{std}}, 1\right)$.
By letting $\epsilon = \epsilon^\text{low}(q)$, the above relation implies \eqref{eq:cor:successprobability:small}. This completes the proof of \eqref{eq:cor:successprobability:small}.
It is trivial to see that $B_\theta^\mathrm{high}(q) \geq 2 \Phi^{-1}(1-q)$ for all $q\in\left( \mathcal{V}_\mathrm{std} , \frac{1}{2}\right)$.

\subsection{Proof of \Cref{lemma:cases}}\label{apdx:lemma:cases}

Proof of (i). 
The 1st term in $\max$ of the lower limit of $I_q$ is smaller than $1/8$. This is because the RHS of \eqref{eq:tr_cond} cannot be greater than or equal to $1/8$; hence, \eqref{eq:tr_cond} implies that 
$\underset{\theta\in\Thetazero}{\sup} \left\{\mathcal{V}_\mathrm{std}\right\} < 1/8$. 
Therefore, it suffices to show that the 2nd term in $\max$ of the lower limit of $I_q$ is smaller than $1/2$.
Namely, we want to show that there exists $\bar{q} < 1/2$ such that $\kappa_{\inf} \cdot \sqrt{\frac{2}{\pi}} > B_{\sup}^\mathrm{low}(q)$ for all $q \in (\bar{q}, 1/2)$.
Let $B_\theta^\text{low}(q;\epsilon)$ as defined in \eqref{eq:b_low_q_e}. 
Because 
$\mathcal{V}_\mathrm{std}< 1/8$,
$B_\theta^\text{low}(q; \epsilon)$ is defined at least for $q \in [1/8, 1/2)$ and $\epsilon \in [1/\sqrt{8 q}, 1)$.
If we let 
\begin{equation}
\delta_1 = \Phi\left(\frac{\kappa_{\inf}}{2}\cdot\frac{1}{\sqrt{2\pi}}\right) - \frac12 - 4\cdot\underset{\theta\in\Thetazero}{\sup}
\mathcal{V}_\mathrm{std} ,
\end{equation}
it is strictly positive because of \eqref{eq:tr_cond}.    
For any $q \in [1/8, 1/2)$, we have
\begin{equation}
\begin{split}
\label{eq:sup_b_low_upper}
&B_{\sup}^\text{low}(q)
\leq
\sup_{\theta \in \Theta_0} B_\theta^\mathrm{low}(q; \epsilon=1/\sqrt{8q})
\\
&=
\sup_{\theta \in \Theta_0}
    \frac{2}{1 - 1/\sqrt{8q}} \Phi^{-1} \left(1 - q + 8 q \cdot \mathcal{V}_\mathrm{std}\right)
\\
&=
\frac{2}{1 - 1/\sqrt{8q}} \Phi^{-1} \left( 1 + 2q \left(\Phi\left(\frac{\kappa_{\inf}}{2}\cdot\frac{1}{\sqrt{2\pi}}\right) - 1 - \delta_1 \right)\right) .
\end{split}
\end{equation}    
Because the RMS of \eqref{eq:sup_b_low_upper} is continuous with respect to $q$ and it approaches $4 \Phi^{-1} \left( \Phi\left(\frac{\kappa_{\inf}}{2}\cdot\frac{1}{\sqrt{2\pi}}\right) - \delta_1 \right) <
4 \Phi^{-1} \left( \Phi\left(\frac{\kappa_{\inf}}{2}\cdot\frac{1}{\sqrt{2\pi}}\right) \right) = \kappa_{\inf}\cdot \sqrt{2 / \pi}$ as $q \to 1/2$, there exists a $\bar{q} \in [1/8, 1/2)$ such that 
$B_{\sup}^\text{low}(q) < \kappa_{\inf}\cdot\sqrt{\frac{2}{\pi}}$ holds for all $q \in (\bar{q}, 1/2)$. 
This completes the proof for (i).

Proof of (ii).  
Let $B_\theta^\mathrm{high}(\cdot;\epsilon)$ as defined in \eqref{eq:b_high_q_e}.
To prove that $I_q^\mathrm{high}(q^\mathrm{low})$ is nonempty, it suffices to show there exists $\bar{q} < 1/2$ such that 
$B_\theta^\mathrm{low}(q^\mathrm{low}) > \frac{\aup}{\adown}\cdot B_\theta^\mathrm{high}(q)$ holds for all $\theta \in \Theta_0$ and for all $q \in (\bar{q}, 1/2)$.
Let $\bar{q} = \Phi\left( - \left(\adown / \aup\right)\cdot \Phi^{-1}\left(1 - q^\mathrm{low} \right) \right)$.
Then $\bar{q} \in (0, 1/2)$. 
Because $B_\theta^\mathrm{low}(q) \geq 2\Phi^{-1}(1-q) \geq B_\theta^\mathrm{high}(q)$ for all $q$ for which $B_\theta^\mathrm{low}(q)$ and $B_\theta^\mathrm{high}(q)$ are defined, we have for all $q \in [\bar{q}, 1/2)$ and for all $\theta \in \Theta_0$
\begin{multline}
\frac{\adown}{\aup} \cdot B_\theta^\mathrm{low}(q^\mathrm{low}) 
\geq 2 \frac{\adown}{\aup} \Phi^{-1}(1-q^\mathrm{low}) 
\\
= 2\Phi^{-1}(1-\bar{q}) 
\geq 2\Phi^{-1}(1-q) 
\geq B_\theta^\mathrm{high}(q) .
\end{multline}
This completes the proof of (ii).


Proof of (iii).
For any $q^\mathrm{low} \in I_q$, we have 
$\kappa_{\inf}\cdot\sqrt{\frac{2}{\pi}} > B_{\sup}^\mathrm{low}(q^\mathrm{low})$. 
Therefore, if there exists $\bar{q} > 0$ such that $B_{\inf}^\text{high}(q) \geq\kappa_{\inf}\cdot\sqrt{\frac{2}{\pi}}$ for all $q \in (0, \bar{q}]$, we have $Q \geq \bar{q} > 0$.
Because $\mathcal{V}_\mathrm{std} < \frac18$ under condition \eqref{eq:tr_cond}, $B_\theta^\mathrm{high}(q; \epsilon)$ is defined for any $q \in (0, 1/2)$ and at least for $\epsilon \in \left[\frac{1}{2\sqrt{1 - 2q}}, 1 \right)$. 
Let 
\begin{equation}
\delta_2 = 
 1 - \Phi\left( \frac{\kappa_{\inf}}{2}\cdot\frac{3}{\sqrt{2\pi}}\right)
 - 4\cdot\underset{\theta\in\Thetazero}{\sup}
\left\{\mathcal{V}_\mathrm{std}  \right\}
 .
\end{equation}
Then, it is positive because of \eqref{eq:tr_cond}. 
For any $q \in (0, 1/2)$, we have
\begin{equation}
\begin{aligned}[t]
\label{eq:inf_b_high_lower}
&B_{\inf}^\text{high}(q)
\geq
\underset{\theta \in \Theta_0}{\inf}\left\{
B_\theta^\text{high}\left(q; \epsilon=\frac{1}{2\sqrt{1 - 2q}}\right) \right\}
\\
&=
\underset{\theta \in \Theta_0}{\inf}
\Bigg\{
\frac{4 \sqrt{1 - 2q}}{2 \sqrt{1 - 2q} + 1} 
\cdot \Phi^{-1} \left(1 - q - 4(1 - 2q) \cdot \mathcal{V}_\mathrm{std}\right)  
\Bigg\}
\\
&=
\frac{4 \sqrt{1 - 2q}}{2 \sqrt{1 - 2q} + 1}
\\
&\quad\cdot \Phi^{-1} \left(1 - q - (1 - 2q) \cdot \left(1 - \Phi\left( \frac{\kappa_{\inf}}{2}\cdot\frac{3}{\sqrt{2\pi}}\right) - \delta_2\right) \right) 
.
\end{aligned}
\end{equation}
Because the RMS of \eqref{eq:inf_b_high_lower} is continuous with respect to $q$ and it approaches $\frac{4}{3} \Phi^{-1} \left( \Phi\left( \frac{\kappa_{\inf}}{2}\cdot\frac{3}{\sqrt{2\pi}}\right) + \delta_2\right) >\kappa_{\inf}\cdot\sqrt{\frac{2}{\pi}}$  as $q \to 0$, there exists a $\bar{q} \in (0, 1/2)$ such that 
$B_{\inf}^\text{high}(q) \geq \kappa_{\inf}\cdot\sqrt{\frac{2}{\pi}}$ holds for all $q \in (0, \bar{q}]$. 
This completes the proof of (iii).

Proof of (iv). 
Condition $\bar\sigma < B_{\inf}^\mathrm{high}(q^\mathrm{high})$ implies $\bar\sigma < B_\theta^\mathrm{high}(q^\mathrm{high})$ for all $\theta\in\Thetazero$.
Then, from \Cref{cor:successprobability}, we immediately find \eqref{eq:lemma:toosmall}.

Proof of (v). 
Condition $\bar\sigma > B_{\sup}^\mathrm{low}(q^\mathrm{low})$ implies
$\bar\sigma > B_\theta^\mathrm{low}(q^\mathrm{low})$ for all $\theta\in\Thetazero$.
Then, from \Cref{cor:successprobability}, we immediately find \eqref{eq:lemma:toolarge}.

Proof of (vi). 
Definition of $Q$ implies that $B_{\sup}^\mathrm{low}(q^\mathrm{low}) \leq B_{\inf}^\mathrm{high}(Q - \xi)$ for any $\xi > 0$. 
Therefore, condition $\bar\sigma \leq B_{\sup}^\mathrm{low}(q^\mathrm{low}) $ implies $\bar\sigma \leq B_{\theta}^\mathrm{high}(Q - \xi)$ for any $\theta\in\Thetazero$.
From \Cref{cor:successprobability}, we find $\Pr\left[f(m + \sigma z) \leq f(m) \right] \geq Q - \xi$. Because $\xi > 0$ is arbitrary, we obtain \eqref{eq:lemma:reasonable}.

\subsection{Proof of \Cref{lemma:potential_decrease}}\label{apdx:lemma:potential_decrease}

In the following proofs, we use abbreviations of the indicator functions as follows:
$\inds = \ind{s\sqrt{ Lf(m_{t+1})} \geq  E_\mathcal{Q}\cdot \sigma_{t+1}}$ and $\indl = \ind{\ell \sqrt{ L f(m_{t+1})} \leq E_\mathcal{Q} \cdot \sigma_{t+1}}$.
Note that $\inds$ and $\indl$ are exclusive to each other, as $s<\ell$.

Let 
\begin{equation}
\begin{aligned}[t]
    \Delta_s &= \log\left(\frac{\aup}{\adown}\right)\inds\inddown
    -\log\left(\aup\right)\cdot\inds
    \\
    &\quad+ \log\left(\frac{s\sqrt{ Lf(m_{t})}}{\sigma_t\cdot E_\mathcal{Q} }\right)\inds
    - \log^+ \left(\frac{s\sqrt{Lf(m_t)}}{\sigma_t \cdot E_\mathcal{Q}}\right),
    \label{eq:delta_s}
\end{aligned}
\end{equation}
\begin{equation}
\begin{aligned}[t]
    \Delta_{\ell} &=
    \log\left(\frac{\aup}{\adown}\right)\indl\indup
    -\log\left(\frac{1}{\adown}\right)\indl
    \\
    &\quad+ \log\left(\frac{\sigma_t\cdot E_\mathcal{Q}}{\ell\sqrt{ Lf(m_{t})}}\right)\indl
    - \log^+ \left(\frac{\sigma_t \cdot E_\mathcal{Q}}{\ell\sqrt{Lf(m_t)}}\right).
    \label{eq:delta_l}
\end{aligned}
\end{equation}
Then, we can write
\begin{multline}
V(\theta_{t+1}) - V(\theta_t)
\\
=
\left(1+(\inds-\indl)\frac{v}{2}\right)\cdot\log\left(\frac{f(m_{t+1})}{f(m_t)}\right)
+ v \cdot \Delta_s + v \cdot \Delta_\ell.
\label{eq:v_decompose}
\end{multline}

\paragraph{Proof of \eqref{eq:general_potential_decrease} for any $\theta\in\Theta\setminus\{\theta\mid m= x^*\}$}
Clearly, the sums of the last two terms of \eqref{eq:delta_s} and \eqref{eq:delta_l} are both non-positive and the second terms of \eqref{eq:delta_s} and \eqref{eq:delta_l} are both non-positive. 
The sum of the first terms of \eqref{eq:delta_s} and \eqref{eq:delta_l} is upper-bounded by $v\cdot\log(\aup / \adown)$ because both $\log(\aup)$ and $\log(1/\adown)$ are positive and $\inds\inddown + \indl\indup \leq 1$.
Because the first term of \eqref{eq:v_decompose} is non-positive, we have
\begin{align}
V(\theta_{t+1}) - V(\theta_t)
\leq& \left(1-\frac{v}{2}\right)\cdot\log\left(\frac{f(m_{t+1})}{f(m_t)}\right)
+v\cdot\log\left(\frac{\aup}{\adown}\right)
\label{eq:hold_above_in_general_case}
.
\end{align}
This completes the proof of \eqref{eq:general_potential_decrease}.

\paragraph{Proof of \eqref{eq:constant_lower_bound_of_V} for any $\theta\in\Theta\setminus\{\theta\mid m= x^*\}$}
Note that $\log^+(a \cdot b) \leq \log^+(a) + \log^+(b)$. This implies that $\log^+(b) - \log^+(a \cdot b) \geq - \log^+(a) \geq -\abs{\log(a)}$. Now, let $a = \frac{\sigma_{t+1} \sqrt{ f(m_{t})}}{ \sigma_{t} \sqrt{f(m_{t+1})} }$ and $b = \frac{s\sqrt{ Lf(m_{t+1})}}{\sigma_{t+1}\cdot E_\mathcal{Q} }$. Then, the above inequality implies 
\begin{equation}
\begin{split}
\MoveEqLeft[2]\log^+\left( \frac{s\sqrt{ Lf(m_{t+1})}}{\sigma_{t+1}\cdot E_\mathcal{Q} } \right) - \log^+\left( \frac{s\sqrt{ Lf(m_{t})}}{\sigma_{t}\cdot E_\mathcal{Q} } \right) 
\\
&\geq - \abs*{\log\left( \frac{\sigma_{t+1} \sqrt{ f(m_{t})}}{ \sigma_{t} \sqrt{f(m_{t+1})} } \right)}
\\
&\geq - \abs*{\log\left(\frac{\sigma_{t+1}}{\sigma_t}\right)} - \frac12 \abs*{\log\left( \frac{ f(m_{t+1}) } { f(m_{t}) }\right)}
\\
&> - \log\left(\frac{\aup}{\adown}\right) + \frac12 \log\left( \frac{ f(m_{t+1}) } { f(m_{t}) }\right) .
\end{split}
\end{equation}
Analogously, we have
\begin{multline}
\log^+\left( \frac{\sigma_{t+1} \cdot E_\mathcal{Q}}{\ell\sqrt{Lf(m_{t+1})}} \right) - \log^+\left( \frac{\sigma_t \cdot E_\mathcal{Q}}{\ell\sqrt{Lf(m_t)}} \right) 
\\
> - \log\left(\frac{\aup}{\adown}\right) + \frac12 \log\left( \frac{ f(m_{t+1}) } { f(m_{t}) }\right) .
\end{multline}
Therefore, we obtain
\begin{equation}
V(\theta_{t+1}) - V(\theta_t) > (1 + v) \cdot \log\left( \frac{ f(m_{t+1}) } { f(m_{t}) }\right) - 2 v \cdot \log\left(\frac{\aup}{\adown}\right) .
\end{equation}

\paragraph{Proof of \eqref{eq:small_sigma_potential_decrease} under $\sigma_t < \frac{s\sqrt{L f(m_t)}}{\aup E_\mathcal{Q}}$}

Under this condition, 
we have $\frac{s\sqrt{Lf(m_t)}}{\sigma_t \cdot E_\mathcal{Q}}>\aup>1$
and $\frac{\sigma_t \cdot E_\mathcal{Q}}{\ell\sqrt{Lf(m_t)}} < \frac{\aup \cdot \sigma_t \cdot E_\mathcal{Q}}{\ell\sqrt{Lf(m_t)}}<\frac{s}{\ell}<1$.
The latter implies that the last term of \eqref{eq:delta_l} is $0$.
Note that the sum of the first three terms of \eqref{eq:delta_l} is $v\cdot\log\left(\frac{\sigma_{t+1}\cdot E_\mathcal{Q}}{\ell\sqrt{ Lf(m_{t})}}\right)\cdot\indl$ and $\sigma_{t+1} \leq \aup \cdot \sigma_t$, implying that the sum of the first three terms of \eqref{eq:delta_l} is non-positive. That is, $\Delta_{\ell} \leq 0$. 
Because the first term of \eqref{eq:v_decompose} is always non-positive, we have $V(\theta_{t+1}) - V(\theta_t) \leq v \Delta_s$ and
\begin{equation}
\begin{aligned}[t]
&v \cdot\Delta_s
\leq
v\cdot\log\left(\frac{\aup}{\adown}\right)\cdot\inds\inddown
- v\cdot\log\left(\aup\right)\cdot\inds
- v \cdot \log \left(\aup\right) \cdot (1 - \inds)
\\
&=
v\cdot\log\left(\frac{\aup}{\adown}\right)\cdot\inds\inddown
- v \cdot \log \left(\aup\right).
\end{aligned}
\end{equation}
By using $\inds \leq 1$, we obtain \eqref{eq:small_sigma_potential_decrease}.

\paragraph*{Proof of \eqref{eq:large_sigma_potential_decrease} under $\sigma_t > \frac{\ell \sqrt{Lf(m_t)}}{\adown  E_\mathcal{Q}}$}

Under this condition, the last term of \eqref{eq:delta_s} is $0$ as $s<\ell$ and $\adown < 1$.
Moreover, we have $\indl = \ind{\ell \sqrt{ L f(m_{t+1})} \leq E_\mathcal{Q} \cdot \sigma_{t+1}} = 1$ and hence $\inds = 0$ because $\sigma_{t+1} \geq \adown\sigma_{t}$ and $\sqrt{ f(m_{t+1})} \leq \sqrt{ L f(m_{t})}$. 
Therefore, the first three terms of \eqref{eq:delta_s} and the sum of the last two terms of \eqref{eq:delta_l} are all $0$. 
Inserting $\indl = 1$ and $\inds = 0$ into \eqref{eq:v_decompose} and \eqref{eq:delta_l}, we obtain
\begin{multline}
V(\theta_{t+1}) - V(\theta_t)
= \left(1-\frac{v}{2}\right)\cdot\log\left(\frac{f(m_{t+1})}{f(m_t)}\right) 
\\
+v\cdot\log\left(\frac{\aup}{\adown}\right)\cdot\indup
-v\cdot\log\left(\frac{1}{\adown}\right)
  ,
\end{multline}
where the first term is always non-positive.
This completes the proof of \eqref{eq:large_sigma_potential_decrease}.

\subsection{Proof of \Cref{cor:expected_potential_decrease}}\label{apdx:cor:expected_potential_decrease}
We divide $\Theta\setminus\{\theta\mid m= x^*\}$ into the following three cases:
 (i) $\sigma_t < \frac{s\sqrt{L f(m_t)}}{\aup E_\mathcal{Q}}$; 
(ii) $\sigma_t > \frac{\ell \norm{\nabla f(m_t)}}{\sqrt{2} \adown  E_z[\mathcal{Q}_z(\theta_t)]}$;
(iii) $\frac{s\sqrt{L f(m_t)}}{\aup E_\mathcal{Q}} \leq \sigma_t \leq \frac{\ell \norm{\nabla f(m_t)}}{\sqrt{2} \adown  E_z[\mathcal{Q}_z(\theta_t)]}$.

In the following, we use the fact $w > 0$. 
From \eqref{eq:w}, it is easy to see that $w > 0$ if $\sqrt{\frac{2}{\pi}}\cdot\kappa_{\inf} > B_{\sup}^\mathrm{low}(q^\mathrm{low})$,  
which follows from $q^\mathrm{low} \in I_q$ and the definition of $I_q$, \eqref{eq:def_Iq}.

\paragraph*{(i) $\sigma_t < \frac{s\sqrt{L f(m_t)}}{\aup E_\mathcal{Q}}$}
In light of \Cref{lemma:potential_decrease}, we have \eqref{eq:small_sigma_potential_decrease}.
Moreover, we have
\begin{multline}
\hspace{-0.8em}
\sigma_t < \frac{s\sqrt{L f(m_t)}}{\aup E_\mathcal{Q}}
= \frac{B_{\inf}^\mathrm{high}(q^\mathrm{high}) \cdot \sqrt{2 L f(m_t)}}{E_\mathcal{Q}}
\\
\leq \frac{B_{\inf}^\mathrm{high}(q^\mathrm{high}) \cdot \norm{\nabla f(m_t)}}{E_\mathcal{Q}}
\leq \frac{B_{\inf}^\mathrm{high}(q^\mathrm{high}) \cdot \norm{\nabla f(m_t)}}{\E[\mathcal{Q}_z(\theta)]}.
\end{multline}
Therefore, the condition of (v) in \Cref{lemma:cases} holds and we have
$\E[\inddown] = 1 - \mathrm{Pr}\left[f(m_t + \sigma_t z_t)\leq f(m_t) \mid \mathcal{F}_t\right] \leq 1 - q^\mathrm{high}$.
By taking the expectation on both sides of \eqref{eq:small_sigma_potential_decrease}, we obtain
\begin{equation}
\begin{split}
\MoveEqLeft[1]\E[ V(\theta_{t+1}) - V(\theta_t) \mid \mathcal{F}_t ]
\\
&\leq
v\cdot \log\left(\frac{\aup}{\adown}\right)\cdot \left( 1 - q^\text{high} - \frac{\log(\aup)}{\log(\aup / \adown)} \right)
\\
&=
v\cdot \log\left(\frac{\aup}{\adown}\right)\cdot \left( p_\mathrm{target} - q^\text{high}\right)
\\
&=
\min\left\{ \frac{w}{4} , \log\left(\frac{\aup}{\adown}\right)\right\}
\cdot \left( p_\mathrm{target} - q^\text{high}\right)
<0
\label{eq:small_expected_potential_decrease}
.
\end{split}
\end{equation}

\paragraph*{(ii) $\sigma_t > \frac{\ell \norm{\nabla f(m_t)}}{\sqrt{2}\adown  E_z[\mathcal{Q}_z(\theta_t)]}$}
This condition implies the condition of (ii) in \Cref{lemma:potential_decrease} as we have $\norm{\nabla f(m_t)} \geq \sqrt{2 L f(m_t)}$. 
Therefore, we have \eqref{eq:large_sigma_potential_decrease}. 
Moreover, we have
\begin{equation}
\bar\sigma_t > \ell / \sqrt{2}\adown
= B_{\sup}^\mathrm{low}(q^\mathrm{low}).
\end{equation}
Therefore, the condition of (vi) in \Cref{lemma:cases} holds and we have
$\E[\indup] = \mathrm{Pr}\left[f(m_t + \sigma_t z_t)\leq f(m_t) \mid \mathcal{F}_t\right] \leq q^\mathrm{low}$.
By taking the expectation on both sides of \eqref{eq:large_sigma_potential_decrease}, we obtain
\begin{multline}
\hspace{-0.8em}
\E[ V(\theta_{t+1}) - V(\theta_t) \mid \mathcal{F}_t ]
\leq v\cdot\log\left(\frac{\aup}{\adown}\right)\cdot \left( q^\mathrm{low} - \frac{\log(1/\adown)}{\log(\aup/\adown)}\right)
\\
= 
\min\left\{ \frac{w}{4} , \log\left(\frac{\aup}{\adown}\right)\right\}
\cdot \left( q^\mathrm{low} - p_\mathrm{target}\right)
<0
\label{eq:large_expected_potential_decrease}
  .
\end{multline}

\paragraph*{(iii) $\frac{s\sqrt{L f(m_t)}}{\aup E_\mathcal{Q}}\leq\sigma_t\leq\frac{\ell \norm{\nabla f(m_t)}}{\sqrt{2}\adown \E[\mathcal{Q}_z(\theta_t)]}$}
Under this condition, we have
$\frac{\sigma_t \E[\mathcal{Q}_z(\theta_t)\cdot\ind{z_e\leq 0}]}{\norm{\nabla f(m_t)}}\leq \frac{\ell}{\sqrt{2}\adown \kappa_{\inf}}$
and
$\frac{\sigma_t \norm{\nabla f(m_t)}}{f(m_t)}\geq \frac{\sigma_t \sqrt{2L}}{\sqrt{f(m_t)}}\geq \frac{s\sqrt{2} L}{\aup E_\mathcal{Q}}$.
In addition, because
\begin{equation}
\bar\sigma_t
\leq
\ell / \sqrt{2}\adown = B_{\sup}^\mathrm{low}(q^\mathrm{low}),
\end{equation}
the condition of (iv) in \Cref{lemma:cases} holds and we have $\Pr[f(m + \sigma z) \leq f(m) ]\geq Q$.
Then,
\begin{equation}
\begin{split}
\MoveEqLeft[1] 
\E\left[ \log\left(\frac{f(m_{t+1})}{f(m_t)}\right)  \mid \mathcal{F}_t \right]
\\
&\leq
\frac{\sigma_t \norm{\nabla f(m_t)}}{f(m_t)}\cdot\left(
\frac{\sigma_t \cdot \E\left[\mathcal{Q}_z(\theta_t)\cdot \ind{z_e\leq 0} \right] }{2\norm{\nabla f(m_t)}}
- \frac{1}{\sqrt{2\pi}}
\right)
\\
&\qquad \cdot \Pr_{z_t}[f(m_t + \sigma_t z_t) \leq f(m_t) ]
\\
&\leq
\frac{s\sqrt{2} L}{\aup E_\mathcal{Q}}
\cdot\left(
- \frac{1}{\sqrt{2\pi}}
+ \frac{\ell}{2\sqrt{2}\adown \kappa_{\inf}}
\right)
\cdot Q
= 
- w 
.
\end{split}
\end{equation}
Taking the expectation on both sides of \eqref{eq:general_potential_decrease}, we obtain
\begin{multline}
\hspace{-0.8em}
\E[ V(\theta_{t+1}) - V(\theta_t)\mid \mathcal{F}_t ]
\leq - \left(1-\frac{v}{2}\right) \cdot w + v\cdot\log\left(\frac{\aup}{\adown}\right)
\\
\leq - \frac{w}{2} + v\cdot\log\left(\frac{\aup}{\adown}\right)
= - \frac{w}{2} + \min \left\{ \frac{w}{4}, \log\left(\frac{\aup}{\adown}\right) \right\}
\leq - \frac{w}{4} .
\label{eq:reasonable_expected_potential_decrease}
\end{multline}

\paragraph*{Altogater}
Finally, by taking a maximum among \eqref{eq:small_expected_potential_decrease}, \eqref{eq:large_expected_potential_decrease} and \eqref{eq:reasonable_expected_potential_decrease}, we obtain $\E[ V(\theta_{t+1}) - V(\theta_t)\mid \mathcal{F}_t ]\leq -B$.
The infimum in $B$ is taken over $q^\mathrm{low}, q^\mathrm{high}\in I_q$ such that $q^\mathrm{low}<p_\mathrm{target}<q^\mathrm{high}$.
The positivity $B$ is guaranteed with $w>0$ and $q^\mathrm{low}<p_\mathrm{target}<q^\mathrm{high}$.
This completes the proof.

\subsection{Proof of \Cref{lemma:v_variance_bound}}\label{apdx:lemma:v_variance_bound}

Note first that
${\color{blue}
\Var[V(\theta_{t+1})]
\leq
\E[\left(V(\theta_{t+1}) - V(\theta_t)\right)^2]}$
%
holds. 
From \eqref{eq:general_potential_decrease} and \eqref{eq:constant_lower_bound_of_V}, 
we obtain
\begin{equation}
\abs{V(\theta_{t+1}) - V(\theta_t)}
\leq
\left(1+v\right)\cdot\left|\log\left(\frac{f(m_{t+1})}{f(m_t)}\right)\right|
+2v\cdot\log\left(\frac{\aup}{\adown}\right)
.
\end{equation}
As the only random variable in the RHS of the expression above is $\left|\log\left(\frac{f(m_{t+1})}{f(m_t)}\right)\right|$,
to prove ${\color{blue}\E[\left(V(\theta_{t+1}) - V(\theta_t)\right)^2]} < \infty$, it suffices to prove that
${\color{blue}\E\left[\abs*{\log\left(\frac{f(m_{t+1})}{f(m_t)}\right)}\right]}
<\infty$
and
${\color{blue}\E\left[\left(\log\left(\frac{f(m_{t+1})}{f(m_t)}\right)\right)^2\right]}
<\infty$.
%
By exploiting the fact that $x \leq \exp(x) - 1$ and $x^2 \leq 2 (\exp(x) - 1)$,
the above inequalities are straightforward for $d>3$ with \Cref{lemma:variancebound}.
This completes the proof. 


%

\section{Biography Section}
 

\begin{IEEEbiography}[{\includegraphics[width=1in,height=1.25in,clip,keepaspectratio]{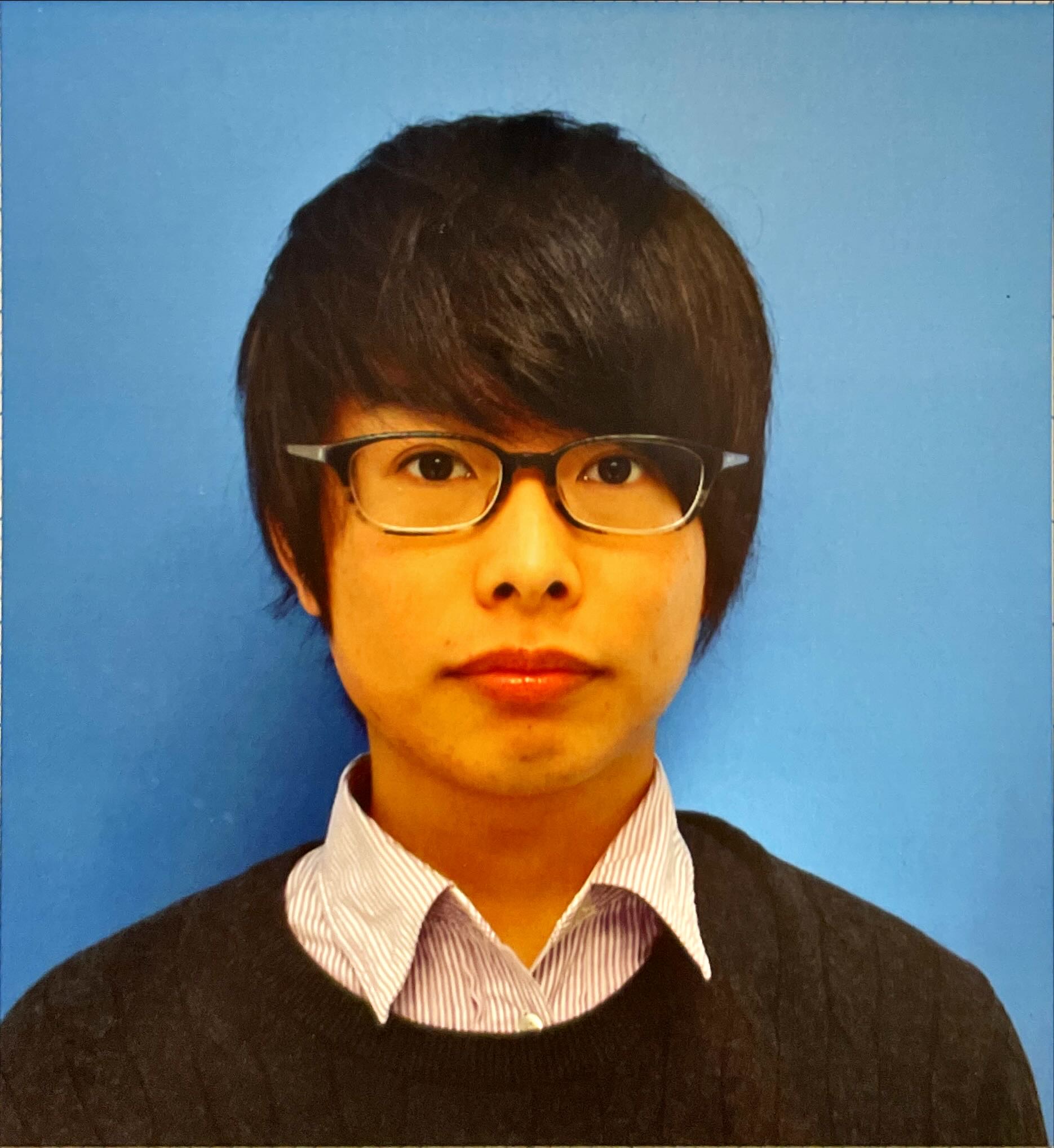}}]{Daiki Morinaga}
received the B.S. degree in information science and the M.S. degree in engineering from University of Tsukuba, Tsukuba, Japan, in 2020 and 2022, respectively.
He won the Best Paper Award at FOGA 2019.
His research interests include optimization techniques in engineering, such as black-box optimization and metaheuristics, and their theoretical analyses.
\end{IEEEbiography}

\begin{IEEEbiography}[{\includegraphics[width=1in,height=1.25in,clip,keepaspectratio]{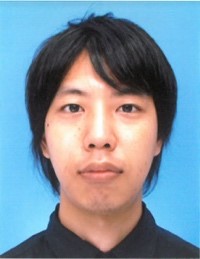}}]{Kazuto Fukuchi} received a Ph.D. degree from University of Tsukuba, Japan, in 2018. Since 2019, he has been an Assistant Professor with the Faculty of Engineering, Information and Systems, University of Tsukuba, Japan, and also a Visiting Researcher with the Center for Advanced Intelligence Project, RIKEN, Japan. His research interests include mathematical statistics, machine learning, and their application.
\end{IEEEbiography}

\begin{IEEEbiography}[{\includegraphics[width=1in,height=1.25in,clip,keepaspectratio]{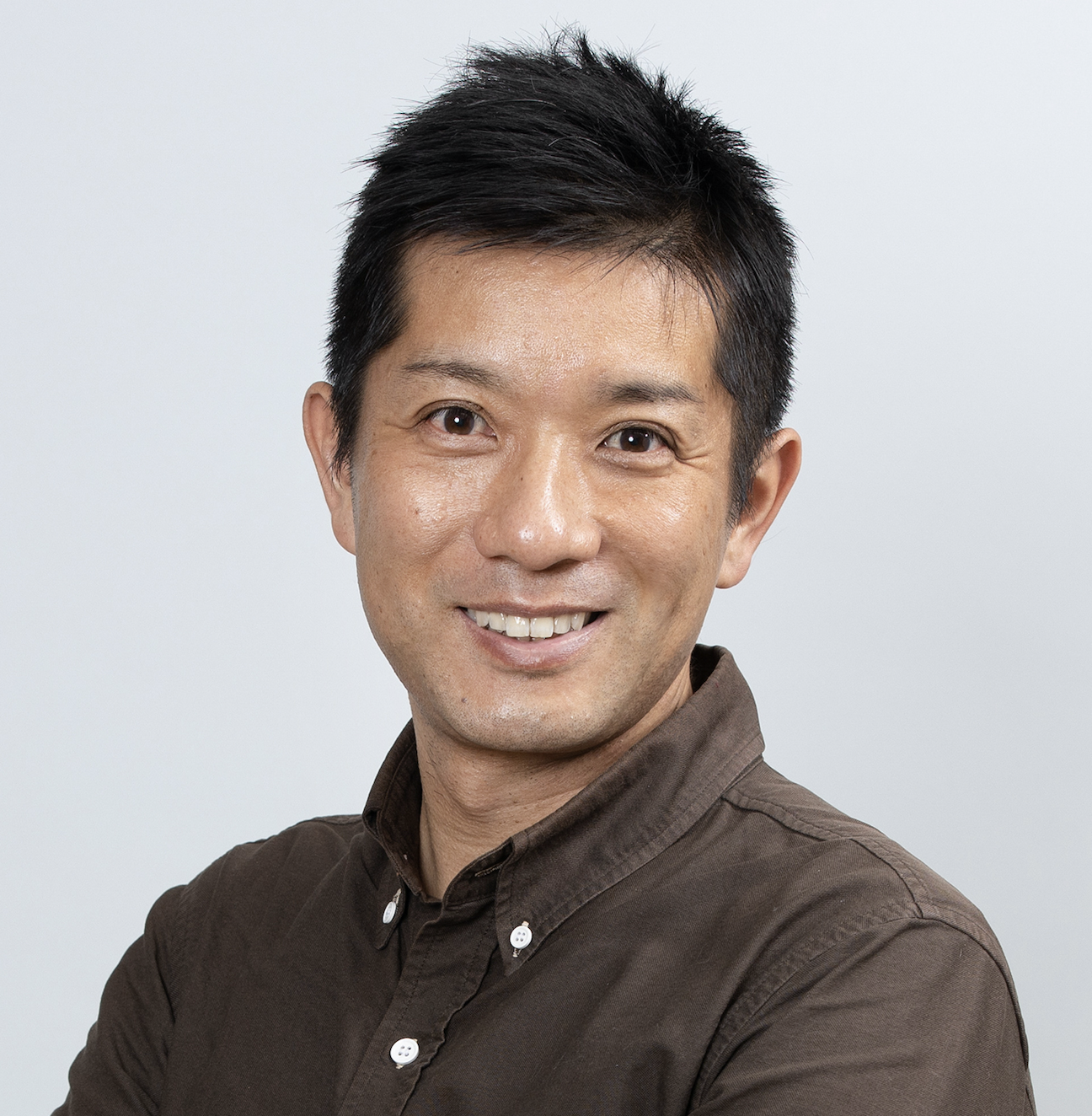}}]{Jun Sakuma} received the Ph.D. degree in engineering from Tokyo
Institute of Technology, Tokyo, Japan, in 2003. Since 2016, he has been a Professor
with the Department of Computer Science, School of Systems and
Information Engineering, University of Tsukuba, Japan. He has also been Team Leader of the Artificial Intelligence Security and Privacy team in the Center for Advanced Intelligence
Project, RIKEN since 2016. From 2009 to 2016, he was an Associate Professor with the Department of Computer Science, University of Tsukuba. From 2004 to 2009, he was an Assistant Professor with the Department of Computational Intelligence and Systems Science, Interdisciplinary Graduate School of Science and Engineering, Tokyo Institute of
Technology, Tokyo, Japan. From 2003 to 2004, he was a Researcher with
Tokyo Research Laboratory, IBM, Tokyo, Japan. His research
interests include data mining, machine learning, data privacy, and
security. He is a member of the Institute of Electronics, Information
and Communication Engineers of Japan (IEICE).
\end{IEEEbiography}

\begin{IEEEbiography}[{\includegraphics[width=1in,height=1.25in,clip,keepaspectratio]{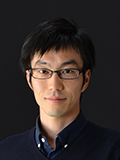}}]{Youhei Akimoto} received the B.S. degree in computer science in 2007, and the M.S. and Ph.D. degrees in computational intelligence and systems science from Tokyo Institute of Technology, Japan, in 2008 and 2011, respectively. From 2010 to 2011, he was a Research Fellow of JSPS in Japan, and from 2011 to 2013, he was a Post-Doctoral Research Fellow with INRIA in France. From 2013 to 2018, he was an Assistant Professor with Shinshu University, Japan. Since 2018, he has been an Associate Professor with University of Tsukuba, Japan as well as a Visiting Researcher with the Center for Advanced Intelligence Projects, RIKEN. 
His research interests include design principles, theoretical analyses, and applications of stochastic search heuristics.
\end{IEEEbiography}



\vfill

\end{document}